\documentclass[a4paper, 10pt]{amsart}
\usepackage{amsthm}
\usepackage[]{amsmath}
\usepackage{amssymb}
\usepackage{enumerate}
\usepackage{tabularx}
\usepackage[]{color}
\usepackage[left=2.2cm, right=2.2cm, top=2.5cm, bottom=2.5cm]{geometry}
\usepackage[colorlinks]{hyperref}
\usepackage{tikz}
\usepackage{multirow}
\usepackage{subcaption}
\usepackage{algorithm}
\usepackage{algorithmic}
\usepackage{xcolor}

\newcommand{\bm}[1]{\boldsymbol{#1}}

\newcommand{\lj}{[ \hspace{-2pt} [}
\newcommand{\rj}{] \hspace{-2pt} ]}
\newcommand{\mb}[1]{\mathbb{#1}}
\newcommand{\mc}[1]{\mathcal{#1}}
\newcommand{\mr}[1]{\mathrm{#1}}
\newcommand{\jump}[1]{\lj #1 \rj}
\newcommand{\aver}[1]{ \{#1\}  }
\newcommand{\wt}[1]{ \widetilde{ #1}}

\newcommand{\DGnorm}[1]{ \| #1\|_{\mr{DG}}}
\newcommand{\DGenorm}[1]{ |\!|\!| #1 |\!|\!|}
\newcommand{\parn}[1]{ \partial_{\un} #1 }
\newcommand{\circi}[1]{ {#1}_{\circ}^i}
\renewcommand{\d}[1]{\mathrm d \boldsymbol{#1}}

\def\curl{\ifmmode \mathrm{curl} \else \text{curl}\fi}
\def\Curl{\ifmmode \mathrm{Curl} \else \text{Curl}\fi}
\def\div{\ifmmode \mathrm{div} \else \text{div}\fi}
\def\Div{\ifmmode \mathrm{Div} \else \text{Div}\fi}
\def\dim{\ifmmode \mathrm{dim} \else \text{dim}\fi}
\def\MTh{\mc{T}_h}
\def\MEh{\mc{E}_h}
\def\MThG{\mc{T}_h^{\Gamma}}
\def\MThB{\mc{T}_h^{\backslash \Gamma}}
\def\MEhG{\mc{E}_h^{\Gamma}}
\def\MEhB{\mc{E}_h^{\backslash \Gamma}}
\def\MThBO{\mc{T}_h^{1, \circ}}
\def\MThBZ{\mc{T}_h^{0, \circ}}
\def\MThBI{\mc{T}_h^{i, \circ}}
\def\lap{\Delta}

\def\un{\bm{\mr n}}

\newcommand\comment[1]{}

\newcommand\substitute[2]{#2}

\newtheorem{assumption}{Assumption}
\newtheorem{theorem}{Theorem}
\newtheorem{lemma}{Lemma}

\newtheorem{remark}{Remark}

\definecolor{orange}{rgb}{1, 0.5, 0}

\allowdisplaybreaks[4]

\title[Biharmonic Interface Problem]{An arbitrary order Reconstructed
Discontinuous Approximation to Biharmonic Interface Problem \\ 
{\scriptsize This paper is dedicated to the Prof. Zhongci Shi}}

\author[Y. Chen]{Yan Chen} \address{School of Mathematical
Sciences, Peking University, Beijing 100871, P.R. China}
\email{yanc@stu.pku.edu.cn}

\author[R. Li]{Ruo Li} \address{CAPT, LMAM and School of Mathematical
Sciences, Peking University, Beijing 100871, P.R. China}
\email{rli@math.pku.edu.cn}

\author[Q.-C. Liu]{Qicheng Liu} \address{School of Mathematical
Sciences, Peking University, Beijing 100871, P.R. China}
\email{qcliu@pku.edu.cn}

\begin{document}

\maketitle

\begin{abstract}
We present an arbitrary order discontinuous Galerkin finite element
method for solving the biharmonic interface problem on the unfitted 
mesh. The approximation space is constructed by a patch reconstruction 
process with at most one degree freedom per element. The discrete 
problem is based on the symmetric interior penalty method and the 
jump conditions are weakly imposed by the Nitsche's technique. 
The $C^2$-smooth interface is allowed to intersect elements in a very 
general fashion and the stability near the interface is naturally 
ensured by the patch reconstruction. We prove the optimal 
\emph{a priori} error estimate under the energy norm and the $L^2$ 
norm. Numerical results are provided to verify the theoretical 
analysis.
\noindent \textbf{keywords}: biharmonic interface problem, 
patch reconstruction,discontinuous Galerkin method
\end{abstract}

\section{Introduction}
\label{sec_introduction}
We are concerned in this paper with the biharmonic interface problem.
The biharmonic operator is a fourth-order elliptic operator which is
frequently seen in the thin plate bending problem and the ions 
transport and distribution problem
\cite{Harari2012embeded, ZienkiewiczTaylorZhu:2015,Liu2023PNPB}. 
Recently, there are many successful finite element methods proposed 
and applied to solve biharmonic problems, see e.g. \cite{Hansbo:2002,
Mozolevski2007hp, cockburn2009hybridizable,
Burman2020cut,Li2019biharmonic}. The biharmonic interface problem
arises in the context of composite materials where the physical 
domain is separated by an interface and the coefficient is 
discontinuous across this interface. Some efforts have been made to 
address this kind of problem
\cite{Nicaise1994biharmonic,Lin2011immersed,Burman2020cut,
Cai2021Nitsche,Cai2023Nitsche}. 

The finite element methods for interface problems can be classified
into two categories: body-fitted methods and unfitted methods. 
Body-fitted methods are constructed on a mesh aligned with the 
interface. This type of methods are naturally suited to to deal with 
the discontinuity on the interface. However, generating a 
high-quality body-fitted mesh can be a nontrivial and time-consuming 
task \cite{Li2020interface, Wu2012unfitted, Liu2020interface}. In 
unfitted methods, the mesh is independent of the interface, and the 
interface can intersect elements in a very general way. As a result, 
unfitted methods have gained more attention for solving the interface 
problem.

In 2002, A. Hansbo and P. Hansbo proposed a Nitsche extended finite
element method (XFEM) for solving the two-order elliptic interface
problem. The idea of this method involves constructing two separated 
finite element spaces on both sides of the interface and using 
Nitsche’s types of penalty to weakly impose the jump conditions. 
Unfitted finite element methods based on this idea are sometimes
referred to as cut finite element methods (CutFEMs). Since then, these 
methods have been further developed and applied to a range of 
interface problems, including Stokes interface problems, 
H(div)-and H(curl)-elliptic interface problems, elasticity interface 
problems, and so on. We refer to \cite{Burman2015cutfem, 
Gurkan2019stabilized, Hansbo2014cut,He2022stabilized,Han2023interface,
Li2023reconstructed, Liu2020interface} and the references therein 
for recent advances.

For the biharmonic interface problem, Y. Cai and et al. proposed two
Nitsche-XFEMs in \cite{Cai2021Nitsche} and \cite{Cai2023Nitsche}. In
\cite{Cai2021Nitsche}, the authors used the so-called modified Morley
finite element to approximate the solution near the interface and
proved an optimal \emph{a priori} error estimate under the energy norm. 
In \cite{Cai2023Nitsche}, they derived a mixed method based on the 
Ciarlet–Raviart formulation with $(P_2, P_2)$ finite element. Due to 
the high order of the biharmonic operator, it is hard to implement a
high-order conforming space. Therefore, we aim to use the
discontinuous Galerkin (DG) method to obtain the numerical solution to
the biharmonic interface problem.

In this paper, we propose an arbitrary order discontinuous Galerkin 
CutFEM for solving the biharmonic problem with a $C^2$-smooth 
interface. The method is based on a reconstructed approximation space
that is constructed by a patch reconstruction procedure. This approach
creates an element patch for each element and solves a local least 
squares fitting problem to obtain a local high-order polynomial
\cite{Li2020interface, Li2016discontinuous, Li2019reconstructed,
Li2019sequential}. Using this new
space, we design the discrete scheme for the biharmonic problem under
the symmetric interior penalty discontinuous Galerkin (IPDG) framework. 

In penalty methods based on unfitted meshes, the small cuts around
the interface may adversely affect the stability of the discrete
system and hamper the convergence. Some stabilized strategies have to
be applied to cure the effects, such as the ghost penalty method and
the extended method, see \cite{Johansson2013high,
Burman2021unfitted, Huang2017unfitted, Burman2010ghost,
Burman2021cutfem, Yang2022an} for some examples.
Benefiting from the reconstruction, we can achieve arbitrary 
approximation accuracy with only one degree of freedom per interior 
element. Besides, we allow the interface to intersect the element in a 
very general way, and the stability in the cut element is ensured
naturally by selecting a proper element patch for it without any extra
stabilized method.

We prove the optimal convergence rates under
the energy norm the $L^2$ norm, respectively. Numerical experiments
are conducted to verify the theoretical analysis and show that our 
algorithm is simple to implement and can reach high-order accuracy.

The rest of this paper is organized as follows. In Section
\ref{sec_preliminaries}, we introduce the biharmonic interface problem
and give the basic notations about the Sobolev spaces and the
partition. We also recall two commonly used inequalities in this
section. In Section \ref{sec_space}, we establish the reconstruction
operator and the corresponding approximation space. Some basic
properties of the reconstruction are also proven in this section. In
Section \ref{sec_interface_problem}, we define the discrete
variational form for the interface problem and analyse the error
under the energy norm and the $L^2$ norm. In Section
\ref{sec_numericalresults}, we carried out some numerical examples to
validate our theoretical results and show high-order accuracy of our
method. Finally, a brief conclusion is given in Section
\ref{sec_conclusion}.


\section{preliminaries}
\label{sec_preliminaries}
Let $\Omega \subset \mb{R}^d (d = 2, 3)$ be a bounded polygonal
(polyhedral) domain with a Lipschitz boundary $\partial \Omega$. Let
$\Gamma$ be a $C^2$-smooth interface that divides the domain $\Omega$ 
into two subdomains $\Omega_0$ and $\Omega_1$, $\Gamma = \partial
\Omega_0 \cap \partial \Omega_1$, see Fig.~\ref{fig_interface_domain}
for an illustration.
\begin{figure}[htp]
  \centering
  \begin{minipage}[t]{0.46\textwidth}
    \begin{center}
      \begin{tikzpicture}
        \draw[thick] (-2.1, -1.5) rectangle (2.1,1.5);
        \draw[thick, fill=white] (0,0) circle [radius =1];
        \node at(0,0) {$\Omega_0$};
        \node at(1.5, 0.8) {$\Omega_1$};
        \draw[thick, ->] (-0.8,0.6) -- (-1.3,1.1);
        \node at(-0.9, 1) {$\un$};
        \node[left] at(-1,0) {$\Gamma$};
        \node[right] at (2.1, 0) {$\Omega$};
      \end{tikzpicture}
    \end{center}
  \end{minipage}
  \begin{minipage}[t]{0.46\textwidth}
    \begin{center}
      \begin{tikzpicture}[scale=1.2]
        \draw[thick] (0, 0) -- (2, 0) -- (2, 2) -- (0, 2) -- (0, 0);
        \draw[thick] (0, 2) -- (0.6, 2.7) -- (2.6, 2.7) -- (2, 2);
        \draw[thick] (2, 0) -- (2.6, 0.7) -- (2.6, 2.7);
        \draw[thick, dashed] (0, 0) -- (0.6, 0.7) -- (2.6, 0.7);
        \draw[thick, dashed] (0.6, 0.7) -- (0.6, 2.7);
        \draw[thick] (1.3, 1.3) circle [radius =0.5];
        \draw[thick] (0.8, 1.3) to [out=-35, in = 215] (1.8, 1.3);
        \draw[thick, dashed] (0.8, 1.3) to [out=35, in = 145] (1.8, 1.3);
        \node[right] at (2.6, 1.7) {$\Omega$};
        \node[below] at (1.3, 0.85) {$\Gamma$};
      \end{tikzpicture}
    \end{center}
  \end{minipage}
  \caption{The sample domain for $d=2$ (left) / $d=3$ (right).}
  \label{fig_interface_domain}
\end{figure}
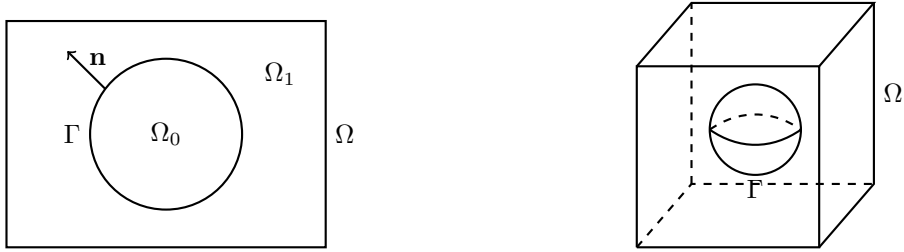
In this paper, we consider the following biharmonic interface problem
\begin{equation}
 \left\{
  \begin{aligned}
    &\lap(\beta \lap u) = f, \quad \text{ in } \Omega_0 \cup \Omega_1, 
    \\
    & u = g_1, \,\, \parn{u} = g_2, \quad \text{ on } \partial \Omega,
    \\
    & \jump{u} = a_1 \un_0,\,\, \jump{\nabla u} = a_2, \\
    & \jump{\beta \lap u} = a_3 \un_0, \quad \qquad \qquad 
    \text{ on } \Gamma, \\
    & \jump{\nabla(\beta \lap u)} = a_4, 
  \end{aligned}
 \right.
 \label{eq_interface}
\end{equation}
where $\beta$ is a positive constant function on $\Omega$ which may be
discontinuous across the interface $\Gamma$. $\jump{\cdot}$ denotes 
the jump operator, which is defined as
\begin{equation}
  \begin{aligned}
\jump{v} &:= v|_{\Omega_0} \un_0 + v|_{\Omega_1} \un_1, \quad \text{ for
scalar-valued function}, \\ 
\jump{\bm{q}} &:= \un_0 \cdot \bm{q}|_{\Omega_0} + 
\un_1 \cdot \bm{q}|_{\Omega_1}, \quad \text{for vector-valued
function},
\end{aligned}
\label{eq_jumpopG}
\end{equation}
where $\un_0$ and $\un_1$ denotes the unit normal on $\Gamma$
orienting from $\Omega_0$ towards $\Omega_1$ and $\Omega_1$ towards
$\Omega_0$, respectively.

Given a bounded domain $D \subset \Omega$, we follow the standard
definitions to the space $L^2(D), L^2(D)^d$, the spaces $H^q(D),
H^q(D)^d$ with the regular exponent $q \geq 0$. For $D_0, D_1 \subset
\mb{R}^d$, we define the Sobolev space $H^q(D_0 \cup D_1)$ as
functions in $D_0 \cup D_1$ such that $u|_{D_i} \in H^q(D_i), \,\,
i=0,1$, with the norm and seminorm
\begin{displaymath}
  \| \cdot \|_{H^q(D_0 \cup D_1)} := \left( \sum_{i=0,1} \| \cdot
  \|_{H^q(D_i)}^2 \right)^{1/2}, \,\, | \cdot |_{H^q( D_0 \cup D_1)} 
  := \left( \sum_{i=0,1} | \cdot |_{H^q(D_i)}^2 \right)^{1/2},
\end{displaymath}
respectively. We assume that under some regular conditions of the data 
$f, g_1, g_2, a_1, a_2, a_3$ and $a_4$, the interface problem 
\eqref{eq_interface} admits an unique solution in 
$H^4(\Omega_0 \cup \Omega_1)$. We refer to \cite{Blum1980boundary} for
detailed regularity results.

We denote by $\MTh$ a regular and quasi-uniform partition $\Omega$
into disjoint open triangles (tetrahedra). The grid is not required to
be fitted to the interface. Let $\MEh$ denote the set of all $d-1$
dimensional faces of $\MTh$, and we decompose $\MEh$ into $\MEh = 
\MEh^\circ \cup \MEh^b$, where $\MEh^\circ$ and $\MEh^b$ are the 
sets of interior faces and boundary faces, respectively. We let 
\begin{displaymath}
  h_K := \text{diam}(K), \quad \forall K \in \MTh, \quad h_e :=
  \text{diam}(e), \quad \forall e \in \MEh, 
\end{displaymath}
and define $h := \max_{K \in \MTh} h_K$.
The quasi-uniformity of the mesh $\MTh$ is in the sense that there
exists a constant $\nu > 0$ such that $h \leq \nu \min_{K \in
\MTh} \rho_K$, where $\rho_K$ is the diameter of the
largest ball inscribed in $K$. One can get the inverse inequality and
the trace inequality from the regularity of the mesh, which are
commonly used in the analysis.
\begin{lemma}
  There exists a constant $C$ independent of the mesh size $h$, such that
  \begin{equation}
    \| v \|^2_{L^2(\partial K)} \leq C \left( h_K^{-1} \| v 
    \|^2_{L^2(K)} + h_K \| \nabla v \|_{L^2(K)}^2 \right), \quad
    \forall v \in H^1(K).
    \label{eq_trace}
  \end{equation}
  \label{le_trace}
\end{lemma}
\begin{lemma}
There exists a constant $C$ independent of the mesh size $h$, such that
  \begin{equation}
    \| v \|_{H^q(K)} \leq C h_K^{-q} \| v \|_{L^2(K)}, \quad 
    \forall v \in \mb{P}_l(K),
    \label{eq_inverse}
  \end{equation}
  where $\mb{P}_l(K)$ is the space of polynomial on $K$ with the 
  degree no more than $l$.
  \label{le_inverse}
\end{lemma}
We refer to \cite{Brenner2007mathematical} for more details of these 
inequalities.

Further, we give the notations related to the interface. For any 
face $e \in \MEh$ and any element $K \in \MTh$, we define 
\begin{displaymath}
  e^i := e \cap \Omega_i, \quad K^i := K \cap \Omega_i, \quad
  (\partial K)^i := \partial K \cap \Omega_i, \quad i = 0, 1,
\end{displaymath}
and we define $\MTh^i$ and $\MEh^i$ as (see Fig.~\ref{fig_MTh0MTh1})
\begin{displaymath}
  \MTh^i := \left\{ K \in \MTh \ | \  |K^i| > 0 \right\}, \quad \MEh^i
  := \left\{ e \in \MEh \ | \  |e^i| > 0  \right\}, \quad i = 0, 1.
\end{displaymath}
We denote by $\MThG$ and $\MEhG$ the set of elements and faces that
are cut by the interface $\Gamma$, respectively, 
\begin{displaymath}
  \MThG := \left\{ K \in \MTh \ | \ K \cap \Gamma \neq \varnothing
  \right\}, \quad \MEhG := \left\{ e \in \MEh \ | \ e \cap \Gamma \neq
  \varnothing \right\}.
\end{displaymath}
We define $\MThB := \MTh \backslash \MThG$ and $\MEhB := \MEh
\backslash \MEhG$, and let $\MThBZ := \MTh^0 \backslash \MThG$,
$\MThBO := \MTh^1 \backslash \MThG$ be the sets of all interior
elements inside the domain $\Omega_0$ and $\Omega_1$, respectively. 
For any cut element $K \in \MThG$, we define $\Gamma_K := K \cap
\Gamma$.  For any element $K$, we define 
\begin{displaymath}
  \partial K^i := \begin{cases}
    \partial K, & K \in \MThBI, \\
    (\partial K)^i  \cup \Gamma_K, & K \in \MThG,\\
  \end{cases} \quad i = 0, 1.
\end{displaymath}

\begin{figure}[htp]
  \centering
  \begin{minipage}[t]{0.33\textwidth}
    \begin{tikzpicture}[scale=2.3]
      \input{figure/mth.tex}
      \draw[red, thick] (0.6, 0) arc [start angle=0, end angle=360,
            x radius=0.6, y radius=0.4]; 
    \end{tikzpicture}
  \end{minipage} 
  \hspace{-0.8cm}
  \begin{minipage}[t]{0.33\textwidth}
    \begin{tikzpicture}[scale=2.3]
      \input{figure/mth1.tex}
      \draw[red, thick] (0.6, 0) arc [start angle=0, end angle=360,
            x radius=0.6, y radius=0.4]; 
    \end{tikzpicture}
  \end{minipage}  
  \hspace{-0.8cm}
  \begin{minipage}[t]{0.33\textwidth}
    \begin{tikzpicture}[scale=2.3]
      \input{figure/mth0.tex}
      \draw[red, thick] (0.6, 0) arc [start angle=0, end angle=360,
            x radius=0.6, y radius=0.4]; 
    \end{tikzpicture}
  \end{minipage}
  \caption{The mesh $\MTh$ (left) / $\MTh^0$ (middle) / $\MTh^1$
  (right), the elements in $\MThG$ (blue).}
  \label{fig_MTh0MTh1}
\end{figure}

We make some geometrical assumptions about the mesh to ensure the
interface is well-resolved by the mesh, which is commonly used in
numerically solving interface problems \cite{Hansbo2002unfittedFEM, 
Cai2023Nitsche}. 

\begin{assumption}
  For any cut face $e \in \MEhG$, $e \cap \Gamma$ is simply connected,
  i.e., the interface does not intersect a face multiple times. 
  \label{as_mesh1}
\end{assumption}
\begin{assumption}
  For any element $K \in \MThG$, there \substitute{exists}{exist} two
  elements $K_\circ^0 \in \MThBZ, K_\circ^1 \in \MThBO$ such that
  $K_\circ^0, K_\circ^1 \in \Delta(K)$, where $\Delta(K)$ denotes the
  Moore neighbours of the element $K$, that is $\Delta(K) := \{ K' \in
  \MTh \ | \ \substitute{K' \cap K \neq \varnothing}{\overline{K'}
  \cap \overline{K} \neq \varnothing} \}.$
  \label{as_mesh2}
\end{assumption}

Such assumptions can always holds true when the mesh is fine enough. 
The Assumption \ref{as_mesh2} allows us to define two maps $M^0(\cdot)$
and $M^1(\cdot)$ such that for any element $K \in \MTh$,
\begin{equation}
  M^i(K) = \begin{cases}
    K, & K \in \MThBI, \\
    K_\circ^i, & K \in \MThG, \\
  \end{cases} \quad i = 0, 1.
  \label{eq_Mi}
\end{equation}
where $K_\circ^i$ is any chosen element that are in the Moore 
neighbours of $K$ and are included in $\Omega_i$. These maps will be 
used in the construction of the reconstruction operator.

Next, we introduce the trace operators that will be used in our 
numerical schemes. For $e \in \MEh^i$, we denote by $K^+$ and $K^-$ the
two neighbouring elements that share the boundary $e$, and $\un^+, \un^-$
the unit out normal vector on $e$, respectively. We define the jump
operator $\jump{\cdot}$ and the average operator $\aver{\cdot}$ as
\begin{equation}
  \begin{aligned}
    \jump{v} &:= v|_{K^+} \un_+ + v|_{K^-} \un_-, \quad \text{ for
    scalar-valued function}, \\
    \jump{\bm{q}} &:= \un^+ \cdot \bm{q}|_{K^+} + \un^- \cdot 
    \bm{q}|_{K^-}, \quad \text{ for vector-valued function}, \\
    \aver{v} &:= \frac{1}{2}(v|_{K^+}  + v|_{K^-}), \quad \text{ for
    scalar-valued function}, \\
    \aver{\bm{q}} &:= \frac{1}{2}(\bm{q}|_{K^+} + \bm{q}|_{K^-}), 
    \quad \text{ for vector-valued function}.
  \end{aligned}
  \label{eq_traceopei}
\end{equation}
For $e \in \MEh^b$, we let $K \in \MTh$ such that $e \in \partial K$ 
and $\un$ is the unit out normal vector. We define 
\begin{equation}
  \begin{aligned}
    &\jump{v} := v|_{K} \un \quad \text{ for scalar-valued function},
    \qquad
    \jump{\bm{q}} := \un \cdot \bm{q}|_{K} \quad 
    \text{ for vector-valued function}, \\
    &\aver{v} := v|_{K}\quad \text{ for scalar-valued function}, \qquad
    \aver{\bm{q}} := \bm{q}|_{K} \quad \text{ for vector-valued
    function}.
  \end{aligned}
  \label{eq_traceopeb}
\end{equation}
For $K \in \MThG$, we have already defined the jump operators in
\eqref{eq_jumpopG}, the average operators are defined as follows.
\begin{equation}
  \begin{aligned}
    \aver{v} &:= \frac{1}{2}(v|_{K_0} + v|_{K_1}), \quad \text{ for 
    scalar-valued function}, \\
    \aver{\bm{q}} &:= \frac{1}{2}(\bm{q}|_{K_0} + \bm{q}|_{K_1}),
    \quad \text{ for vector-valued function}.
  \end{aligned}
  \label{eq_averageG}
\end{equation}

Throughout this paper, $C$ and $C$ with subscripts denote the generic 
constants that may differ between lines but are independent of the 
mesh size and how the interface intersects the mesh.


\section{Reconstructed Discontinuous Space}
\label{sec_space}
In this section, we aim to construct the reconstructed discontinuous 
space for approximating the problem \eqref{eq_interface} by using a 
global reconstruction operator. The global operator which we denote 
by $\mc{R}$ has two components $\mc{R}^0$ and $\mc{R}^1$ with respect 
to the sets $\MTh^0$ and $\MTh^1$. To obtain these operators, we employ 
a reconstruction process for each $\MTh^i,\,\,i=0,1$ that consists of 
three steps. First, we mark the barycenter $\bm{x}_K$ for every 
interior element $K \in \MThBI$ as a collocation point, $i = 0,1$, 
see Fig.~\ref{fig_collocationpoint}.
\begin{figure}[htp]
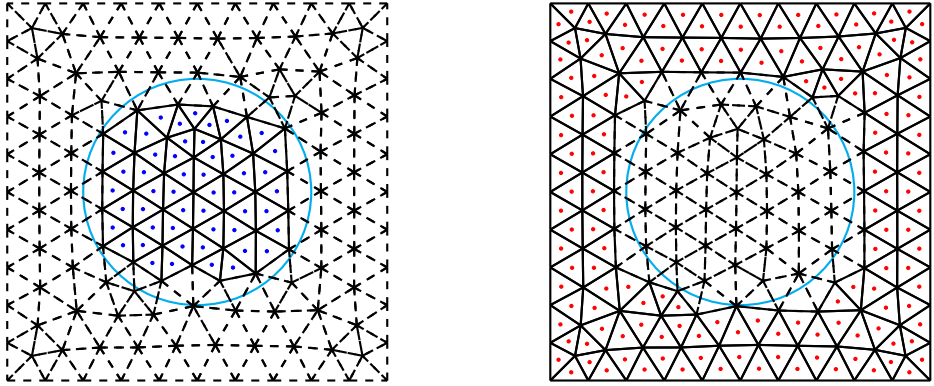

  \centering
  \begin{minipage}[tb]{0.45\textwidth}
    \centering
    \begin{tikzpicture}[scale=2.5]
      \draw[cyan, thick] (0, 0) circle [radius = 0.6];
      \input{figure/mth0p2_2.tex}
    \end{tikzpicture}
    \hspace{-1cm}
  \end{minipage} \begin{minipage}[tb]{0.45\textwidth}
    \centering
    \begin{tikzpicture}[scale=2.5]
      \draw[cyan, thick] (0, 0) circle [radius = 0.6];
      \input{figure/mth1p2_2.tex}
    \end{tikzpicture}
  \end{minipage}
  \caption{The collocation points in $\MTh^0$ (left) / $\MTh^1$
  (right). }
  \label{fig_collocationpoint}
\end{figure}

The second step is to construct the element patch for all elements.
For any element $K \in \MTh^i$, we need to construct a patch set
$S^i(K)$ which contains several elements that are also in $\MTh^i$.
For any $K \in \MThBI$, the $S^i(K)$ is selected by a recursive
algorithm. We assign a threshold $\# S$ to be the cardinality of the
element patch and begin the recursion by setting $S^i_0(K) = \{ K \}$,
and then define $S_t^i(K)$ recursively:
\begin{displaymath}
  S_t^i(K) := \bigcup_{K' \in S_{t - 1}^i(K)} \ \bigcup_{K'' \in
  \Delta(K'), \ K'' \in \MTh^i} K'', \quad
  t=0,1,\ldots
\end{displaymath}
The recursion terminates once $t$ satisfies the condition $\# S_t^i(K) 
\geq \# S$. Then we sort the elements in $S_t^i(K)$ according to the 
distance between their barycenter and the barycenter of $K$, and 
select the $\# S$ elements with the shortest distance to form the 
patch $S^i(K)$. For the element $K \in \MThG$, we need to construct 
two element patches $S^0(K) \subset \MTh^0$ and $S^1(K) \in \MTh^1$. 
From Assumption \ref{as_mesh2}, we have $M^0(K) \in \MThBZ$ and 
$M^1(K) \in \MThBO$, and we directly let $S^i(K) = S^i(M^i(K))$. 
Here we assume $K \in S^i(M^i(K))$ and this can be easily fulfilled 
for a bit large $\# S$. We denote by $I^i(K)$ the set of all 
collocation points located in $S^i(K)$,
\begin{displaymath}
  I^i(K) := \{ \bm{x}_{\wt{K}} | \wt{K} \in S^i(K) \cap \MThB \}.
\end{displaymath}

The final step is to solve a local constrained least squares fitting 
problem. The least squares problem resolves a polynomial of 
degree $m$ from a piecewise constant function in $U_h^0$, where 
\begin{displaymath}
  U_h^0 := \{v_h \in L^2(\Omega) \ | \ v_h|_K \in \mb{P}_0(K), \quad
  \forall K \in \MTh \},
\end{displaymath}
Given a function $g_h \in U_h^0$ and an integer $m \geq 1$, we consider
the following problem for each element $K \in \MTh^i, \,\, i =0,1$:
\begin{equation}
  p_{S^i(K)} = \mathop{\arg \min}_{ q \in
  \mb{P}_m(S^i(K))} \sum_{\bm{x} \in I^i(K)} | q(\bm{x}) -
  g_h(\bm{x}) |^2, \quad \text{s.t. } 
  q(\bm{x}_{M^i(K)}) = g_h(\bm{x}_{M^i(K)}).
  \label{eq_lsproblem}
\end{equation}
The constraint in \eqref{eq_lsproblem} is crucial for proving the 
linear independence result in Lemma \ref{le_lambdaid}. 
We make the following geometrical assumption on the location of 
collocation points \cite{Li2012efficient, Li2016discontinuous}: 
\begin{assumption}
  For any element patch $S^i(K)$ and any polynomial $p \in
  \mb{P}_m(S^i(K))$, $p|_{I^i(K) } = 0$ implies $p|_{S^i(K)} = 0$.
  \label{as_lsproblem}
\end{assumption}
This assumption excludes the case that the points in $I^i(K)$ are on 
an algebraic curve of degree $m$ and requires $\# I^i(K) \geq 
\dim(\mb{P}_m(\cdot))$. Under this assumption, the fitting problem 
\eqref{eq_lsproblem} admits a unique solution. We note that $p_{S^i(K)}$
depends linearly on the given function $g_h$ since it is obtained by a
least squares problem. This property inspires us to define a linear 
local reconstruction operator $\mc{R}_K^i$ for all elements in $\MTh^i$ 
by restricting the $p_{S^i(K)}$ to $K$,
\begin{displaymath}
  \begin{aligned}
    \mc{R}_K^i : U_h^0  &\rightarrow \mb{P}_m(K), \\
    g_h & \rightarrow (p_{S^i(K)})|_K, \\
  \end{aligned} \quad \forall K \in \MTh^i, \quad i=0,1.
\end{displaymath}
Based on the local operators, we can further define two
reconstruction operators piecewise as follows
\begin{displaymath}
  \begin{aligned}
    \mc{R}^i : U_h^0 & \rightarrow U_h^{m, i}, \\
    g_h & \rightarrow \mc{R}^i g_h, \\
  \end{aligned} \quad
  (\mc{R}^i g_h)|_K := \mc{R}_K^i g_h,   
  \quad \forall K \in \MTh^i, \quad i=0,1,
\end{displaymath}
where $U_h^{m,i}$ represents the image space of the operator 
$\mc{R}^i$. Obviously, $U_h^{m,i}$ is a subspace of the space of 
piecewise polynomials of degree $m$ over the partition $\MTh^i$. Next,
we investigate the basis functions of $U_h^{m,i}$. For any element 
$K \in \MTh^i$, we pick up a function $e_K \in U_h^0$ such that
\begin{displaymath}
  e_K(\bm{x}) = \begin{cases}
    1, & \bm{x} \in K, \\
    0, & \text{otherwise}. \\
  \end{cases}
\end{displaymath}
We let $\lambda_K^i := \mc{R}^i e_K, \,\, K \in \MThBI$ and we state 
that the space $U_h^{m,i}$ is spanned by $\{ \lambda_K^i \}$.
\begin{lemma}
  For $i = 0, 1$, the functions $ \{ \lambda_K^i \}(K \in
  \MThBI)$ are linearly independent and the space $U_h^{m, i} = 
  \text{span}( \{ \lambda_{K}^i \})$. 
  \label{le_lambdaid}
\end{lemma}
\begin{proof}
  For any $K \in \MThBI$, assume there exists a group of coefficients
  $\{a_K^i\}$ such that  
  \begin{equation}
    \sum_{K \in \MThBI} a_{K}^i
    \lambda_{K}^i(\bm{x}) = 0, \quad \forall \bm{x} \in
    \mb{R}^d.
    \label{eq_alambda}
  \end{equation}
  We take $\bm{x} = \bm{x}_K$ for all $K \in \MThBI$ in 
  \eqref{eq_alambda}, from the constraint in the problem 
  \eqref{eq_lsproblem}, we have that 
  \begin{displaymath}
      \lambda_{K}^i(\bm{x}_{K'}) = \begin{cases}
      1, & K' = K, \\
      0, & \text{otherwise}, \\
    \end{cases}
  \end{displaymath}
  which infers that $a_K^i = 0$ for all $K \in \MThBI$. Thus,
  the functions $ \{ \lambda_{K}^i \}(K \in \MThBI)$ are linearly 
  independent. For any function $g_h \in U_h^0$, it can be decomposed 
  as
  \begin{displaymath}
    g_h = \sum_{K \in \MTh^i} g_h(\bm{x}_{K}) e_K.
  \end{displaymath}
  Thus, one can explicitly write $\mc{R}^i g_h$ as
  \begin{equation}
    \mc{R}^i g_h = \sum_{K \in \MTh^i} g_h(\bm{x}_K) \mc{R}^i e_K
    = \sum_{K \in \MThBI} g_h(\bm{x}_K) \lambda_K^i,
    \label{eq_Rig}
  \end{equation}
  since $\mc{R}^i e_K = 0, \quad \forall K \in \MThG$. Then we can 
  conclude that $U_h^{m,i}$ is spanned by $\{ \lambda_K^i \}$, which 
  completes the proof.
\end{proof}
From the problem \eqref{eq_lsproblem}, the basis function
$\lambda_K^i$ vanishes on the element $K'$ that $K \not\in
S^i(K')$.  This fact indicates $\lambda_K^i$ has a finite
support set that $\text{supp}(\lambda_K^i) = \{K' \in \MTh^i
\ | \ K \in S^i(K') \}$. Fig.~\ref{fig_basis} presents two examples of
the basis function. We can extend the operator $\mc{R}^i(i = 0,1)$
to act on smooth functions. For any $g \in H^{m+1} (\Omega)$, we define
a piecewise constant function $g_h \in U_h^0$ as
\begin{displaymath}
  g_h(\bm{x}_K) := g(\bm{x}_K),
\end{displaymath}
and we directly define $\mc{R}^i g := \mc{R}^i g_h$. 
In this way, any smooth function in $H^{m+1}(\Omega)$ is mapped into a
piecewise polynomial function with respect to $\MTh^i$ by the operator
$\mc{R}^i$ and for any $g \in H^{m+1}(\Omega)$, $\mc{R}^i g$
can also be written as \eqref{eq_Rig}. 
\begin{figure}[htbp]
  \centering
  \includegraphics[width=0.3\textwidth]{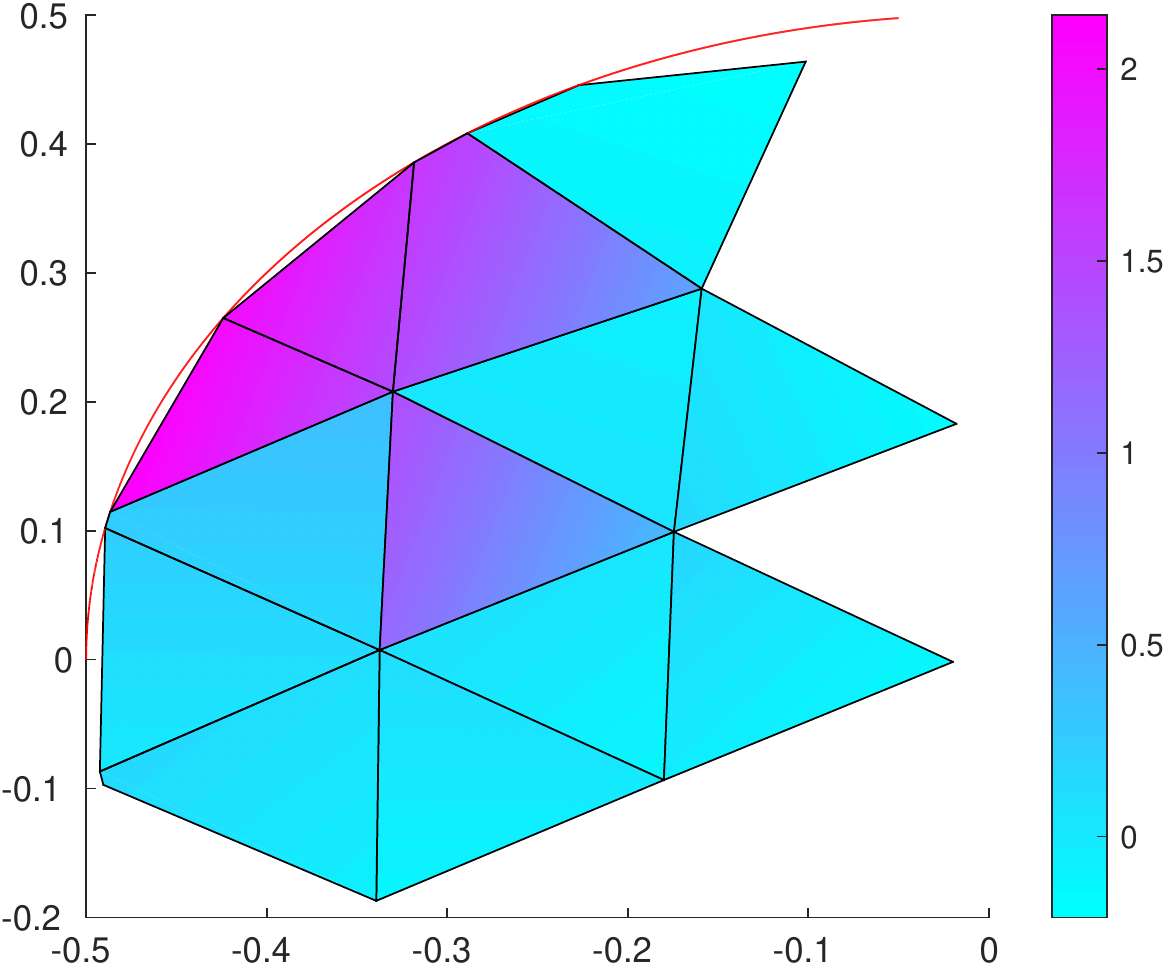}
  \hspace{1cm}
  \includegraphics[width=0.3\textwidth]{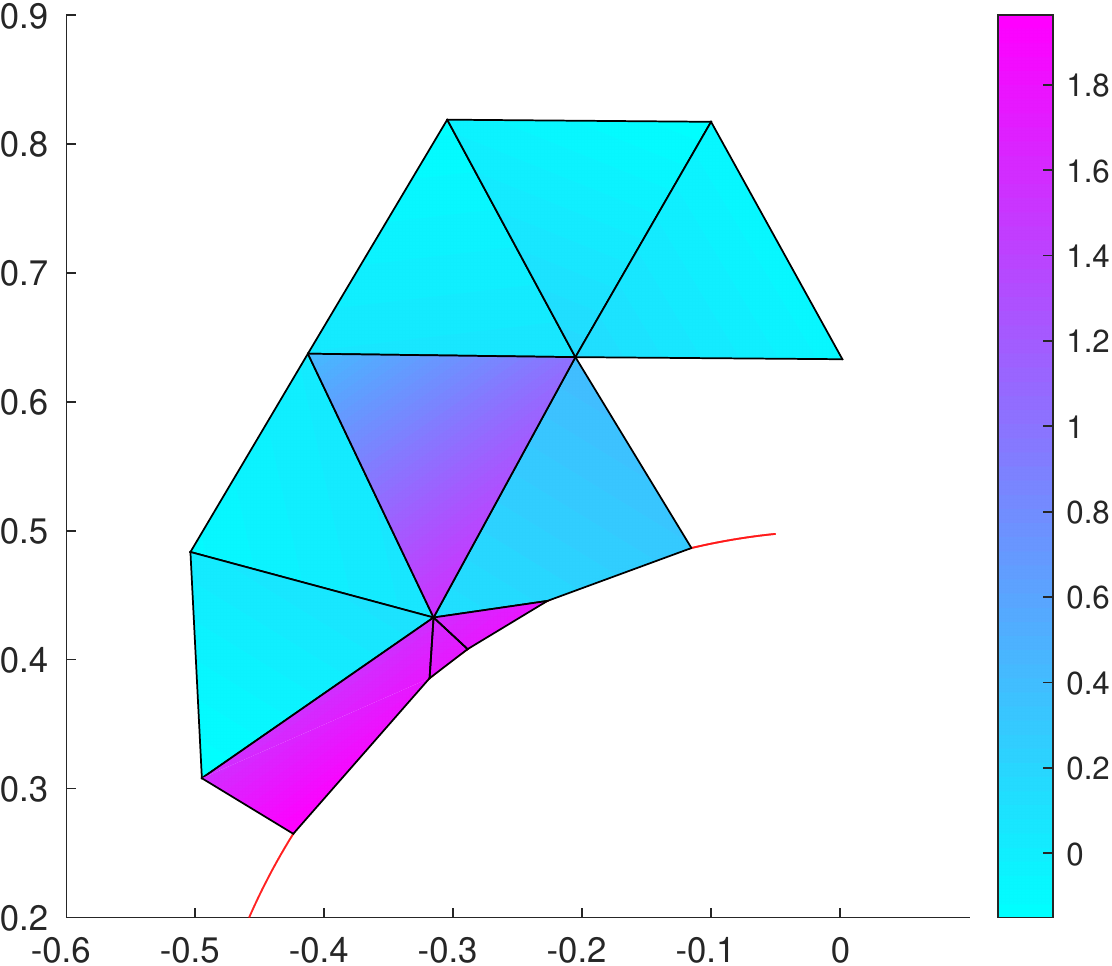}
  \caption{Two basis functions inside the interface (left) and outside
  the interface (right).}
  \label{fig_basis}
\end{figure}

Next, we will prove some approximation properties of the operator
$\mc{R}^i(i = 0, 1)$. We first define a constant $\Lambda(m, S^i(K))$
for every element patch \cite{Li2016discontinuous},
\begin{displaymath} 
  \Lambda(m, S^i(K)) := \max_{p \in
  \mb{P}_m(S^i(K))} \frac{\max_{\bm{x} \in S^i(K)} |p(\bm{x})|
  }{\max_{\bm{x} \in I^i(K)} |p(\bm{x})| }.
\end{displaymath}
Assumption \ref{as_lsproblem} actually ensures $\Lambda(m, S^i(K)) <
\infty$, and we set 
\begin{equation}
  \Lambda_m := \max_{i = 0, 1} \max_{K \in \MTh^i} \left( 1 +
  \Lambda(m, S^i(K)) \sqrt{\# I^i(K)} \right).
  \label{eq_Lambdam}
\end{equation}
\begin{lemma}
  For any element $K \in \MTh^i$, there holds
  \begin{equation}
    \|g - \mc{R}_K^i g \|_{L^{\infty}(K)} \leq 2
    \Lambda_m \inf_{q \in \mb{P}_m(S^i(K))} \| g - q
    \|_{L^\infty (S^i(K))}, \quad \forall g \in H^{m+1}(\Omega). 
    \label{eq_stability}
  \end{equation}
  \label{le_stability}
\end{lemma}
\begin{proof}
  We define a polynomial space 
  \begin{displaymath}
    \wt{\mb{P}}_m(S^i(K)) := \left\{ v \in \mb{P}_m(S^i(K))
    \ | \  v(\bm{x}_{M^i(K)}) = g(\bm{x}_{M^i(K)}) \right\}.
  \end{displaymath}
  Clearly, any polynomial in $\wt{\mb{P}}_m(S^i(K))$ satisfies the
  constraint in \eqref{eq_lsproblem}. We let $p := \mc{R}^i_K g$
  , for any $\varepsilon \in \mb{R}$ and any 
  $q \in \wt{\mb{P}}_m(S^i(K))$, we have that $p + \varepsilon 
  ( q - g(\bm{x}_{M^i(K)})) \in \wt{\mb{P}}_m(S^i(K))$ and 
  \begin{displaymath}
    \sum_{\bm{x} \in I^i(K)} | p(\bm{x}) + \varepsilon (
    q(\bm{x}) - g(\bm{x}_{M^i(K)}))- g(\bm{x})
    |^{2} \geq \sum_{\bm{x} \in I^i(K)} | p(\bm{x}) -
    g(\bm{x}) |^{2}.
  \end{displaymath}
  Because $\varepsilon$ is arbitrary, the above inequality implies 
  \begin{displaymath}
    \sum_{\bm{x} \in I^i(K)} (p(\bm{x}) - g(\bm{x})) \cdot
    (q(\bm{x}) - g(\bm{x}_{M^i(K)})) = 0,
  \end{displaymath}
  for any $q \in \wt{\mb{P}}_m(S^i(K))^d$. By letting $q =
  p$ and applying Cauchy-Schwarz inequality, 
  \begin{displaymath}
    \begin{aligned}
      0 &= \sum_{\bm{x} \in I^i(K)} (g(\bm{x}) -
      g(\bm{x}_{M^i(K)}) + g(\bm{x}_{M^i(K)}) - p(\bm{x}))
      \cdot (p(\bm{x}) - g(\bm{x}_{M^i(K)}))
      \\
      & = \sum_{\bm{x} \in I^i(K)} \left( -| p(\bm{x}) -
      g(\bm{x}_{M^i(K)}) |^2 + (p(\bm{x}) -
      g(\bm{x}_{M^i(K)})) \cdot (g(\bm{x}) - g(\bm{x}_{M^i(K)})) 
      \right) \\
      & \leq -\frac{1}{2} \sum_{\bm{x} \in I^i(K)} | p(\bm{x}) -
      g(\bm{x}_{M^i(K)}) |^2 + \frac{1}{2} \sum_{\bm{x} \in I^i(K)}
      | g(\bm{x}) - g(\bm{x}_{M^i(K)}) |^2,
    \end{aligned}
  \end{displaymath}
  and we obtain
  \begin{equation}
    \sum_{\bm{x} \in I^i(K)} | \mc{R}^i_K g -
    g(\bm{x}_{M^i(K)}) |^2 \leq  \sum_{\bm{x} \in I^i(K)} 
    | g(\bm{x}) - g(\bm{x}_{M^i(K)}) |^2.
    \label{eq_pgh1}
  \end{equation}
  Moreover, for any $w \in {\mb{P}}_m(S^i(K))$, it is trivial
  to see $\mc{R}_K^i w = w$, and we derive that 
  \begin{displaymath}
    \begin{aligned}
      \| g - \mc{R}_K^i g \|_{L^{\infty}(K)} = \|g - w + w - 
      \mc{R}_K^i g \|_{L^{\infty}(K)} \leq
      \|g - w \|_{L^{\infty}(K)} +  \| w - \mc{R}_K^i g 
      \|_{L^{\infty}(K)},
    \end{aligned}
  \end{displaymath}
  and 
  \begin{displaymath}
    \begin{aligned}
      \| w - \mc{R}_K^i g \|_{L^{\infty}(K)} &= \| \mc{R}_K^i
      (w - g) \|_{L^{\infty}(K)} \\
      &\leq \| \mc{R}_K^i (w - g) - (w - g)(\bm{x}_{M^i(K)}) 
      \|_{L^\infty(K)} + | (w - g)(\bm{x}_{M^i(K)}) |. 
    \end{aligned}
  \end{displaymath}
  From \eqref{eq_pgh1} and the constant $\Lambda(m, S^i(K))$, it can
  be seen that
  \begin{displaymath}
    \begin{aligned}
      \| \mc{R}_K^i (w - g) - & (w - g)(\bm{x}_{M^i(K)})
      \|_{L^\infty(K)} \leq \Lambda(m, S^i(K)) \max_{\bm{x} \in
      I^i(K)} | \mc{R}_K^i (w - g)(\bm{x}) - (w-g)(\bm{x}_{M^i(K)})
      | \\
      &\leq \Lambda(m, S^i(K)) \left( \sum_{\bm{x} \in I^i(K)}  |
      \mc{R}_K^i (w - g)(\bm{x}) - (w - g)(\bm{x}_{M^i(K)}) |^2 
      \right)^{1/2} \\
      & \leq \Lambda(m, S^i(K)) \sqrt{\# I^i(K)} \max_{
      \bm{x} \in I^i(K)} \| (w - g)(\bm{x}) - (w - g)(\bm{x}_{M^i(K)}) 
      | \\
      & \leq 2\Lambda(m, S^i(K)) \sqrt{\# I^i(K)} \|w - g
      \|_{L^\infty(S^i(K))}.
    \end{aligned}
  \end{displaymath}
  Combining all above inequalities, we conclude that
  \begin{displaymath}
    \begin{aligned}
      \| g - \mc{R}^i_K g \|_{L^{\infty}(K)} &\leq 
      2\Lambda(m, S^i(K)) \sqrt{ \# I^i(K)} \| w - g
      \|_{L^{\infty}(S^i(K))} \\
      &+ \| w - g \|_{L^{\infty}(K)} + | (w - g)(\bm{x}_{M^i(K)}) 
      | \\
      &\leq 2(1 + \Lambda(m, S^i(K)) \sqrt{ \# I^i(K)}) \| w - g
      \|_{L^{\infty}(S^i(K))}, 
    \end{aligned}
  \end{displaymath}
  for any $w \in \mb{P}_m(S^i(K))$, which implies \eqref{eq_stability}
  and completes the proof.
\end{proof}
From the stability result \eqref{eq_stability}, we can prove the 
following approximation property.
\begin{lemma}
  For any $K \in \MTh^i(i = 0,1)$ and $g \in H^{m+1}(\Omega)$, there 
  exists a constant $C$ such that
  \begin{equation}
    \| g - \mc{R}^i g \|_{H^q(K)} \leq C \Lambda_m h_K^{m+1-q} \| g
    \|_{H^{m+1}(S^i(K))}, \quad 0 \leq q \leq m.
    \label{eq_approximation}
  \end{equation}
  \label{le_approximation}
\end{lemma}
\begin{proof}
  By Lemma \ref{le_stability}, we have
  \begin{displaymath}
    \begin{aligned}
    \| g - \mc{R}^i g \|_{L^2(K)} &\leq |K|^{1/2} \| g - \mc{R}^i g
    \|_{L^\infty(K)} \leq C h_K^{d/2} \Lambda_m \inf_{p \in 
    \mb{P}_m(S^i(K))} \|g - p \|_{L^\infty(S^i(K))} \\
    &\leq C \Lambda_m h^{m+1} \| g \|_{H^{m+1}(S^i(K))}.
    \end{aligned}
  \end{displaymath}
  The case $q > 0$ will be proved by using the inverse 
  inequality \eqref{eq_inverse}. Take a polynomial $\wt{p} \in
  \mb{P}_m(S^i(K))$, we can obtain that
  \begin{displaymath}
    \begin{aligned}
    \| g - \mc{R}^i g \|_{H^q(K)} &\leq \| g - \wt{p} \|_{H^q(K)} + \|
    \wt{p} - \mc{R}^i g \|_{H^q(K)} \\
    &\leq C \inf_{p \in \mb{P}_m(S^i(K))} \| g - p \|_{H^q(S^i(K))} + 
    Ch_K^{-q} \| \wt{p} - \mc{R}^i g \|_{L^2(K)} \\
    & \leq C h_K^{m+1-q} \| g \|_{H^{m+1}(S^i(K))} + C h_K^{-q} \| g -
    \mc{R}^m g \|_{L^2(K)} \\
    & \leq C h_K^{m+1-q} \| g \|_{H^{m+1}(S^i(K))},
    \end{aligned}
  \end{displaymath}
  which completes the proof.
\end{proof}
\begin{remark}
  In the proof of Lemma \ref{le_approximation}, we have used the fact
  that for any $g \in H^{m+1}(S^i(K))$, there exists a polynomial $p
  \in \mb{P}_m(S^i(K))$ such that
  \begin{displaymath}
    \begin{aligned}
      \| g - p \|_{H^q(S^i(K))} &\leq C h_K^{m+1-q} \| g
      \|_{H^{m+1}(S^i(K))}, \quad 1 \leq  q \leq m+1, \\
      \| g - p \|_{L^\infty(S^i(K))} &\leq C h_K^{m+1 - 2/d}
      \|g\|_{H^{m+1}(S^i(K))}.
    \end{aligned}
  \end{displaymath}
  These hold in condition that $S^i(K)$ is star-shaped. We refer to
  \cite{Dupont1980polynomial} for the details.
\end{remark}
From \eqref{eq_approximation}, it can be seen that the operator
$\mc{R}^i$ has an approximation error of degree $O(h^{m+1-q})$
provided that $\Lambda_m$ admits an upper bound independent of the
mesh size $h$. Under some mild assumptions about the element patch,
we can prove that the constant $\Lambda(m, S^i(K))$ has a uniform
upper bound.
\begin{assumption}
  For every element patch $S^i(K)(K \in \MTh, i = 0 \text{ or } 1)$,
  there exist constants $R$ and $r$ which are independent of $K$ such
  that $B_r \subset S^i(K) \subset B_R$, and $S^i(K)$ is star-shaped
  with respect to $B_r$, where $B_\rho$ is a disk with the radius
  $\rho$. 
  \label{as_patch}
\end{assumption}
Under the geometrical Assumption \ref{as_patch}, we have the 
following result.
\begin{lemma}
  For any $\varepsilon > 0$, if 
  \begin{displaymath}
    r > m \sqrt{2 C_{S^i(K)} R h (1 + 1/ \varepsilon)}, 
    \qquad
    C_{S^i(K)} =
    \begin{cases}
      3, & \text{$S^i(K)$ is cut by $\Gamma$}, \\
      1, & \text{otherwise}, \\
    \end{cases}
  \end{displaymath}
  then 
  \begin{equation}
    \Lambda(m, S^i(K)) \leq 1 + \varepsilon,
    \label{eq_LmS}
  \end{equation}
  for any patch $S^i(K)$. Particularly, if $r > 2m \sqrt{C_{S^i(K)} 
  R h}$, $\Lambda(m, S^i(K)) \leq 2$.
  \label{le_Lambdam}
\end{lemma}
\begin{proof}
  For the proof of this lemma the reader is referred to
  \cite{Li2023reconstructed}. 
\end{proof}
According to Lemma \ref{le_Lambdam}, to ensure the stability of the 
reconstruction operator, we are required to construct a large element 
patch to make the radius r as large as possible. The geometrical
conditions can be fulfilled if the threshold $\# S$ is larger than a
certain constant which only depends on the mesh and $m$. We refer to 
\cite[Lemma 6]{Li2016discontinuous} and 
\cite[Lemma 3.4]{Li2012efficient} for the details of this statement.
We list the value of $\# S$ in our numerical experiments in Section
\ref{sec_numericalresults}.

Finally, we define a global reconstruction operator $\mc{R}$ by
combining two operators $\mc{R}^0$ and $\mc{R}^1$. We choose two
extension operators which extend the functions in $H^q(\Omega_i)$ to
$H^q(\Omega),\,\,q \geq 1$. From \cite[Chapter 5]{Adams2003sobolev}, 
we know that for any $w \in H^q(\Omega_0 \cup \Omega_1)$,
there exist two operators $E^i : H^q(\Omega_i) \rightarrow
H^q(\Omega)$ such that 
\begin{displaymath}
  (E^i w)|_{\Omega_i} = w, \quad \| E^i w \|_{H^s(\Omega)} \leq C \| w
  \|_{H^s(\Omega_i)}, \quad 0 \leq s \leq q, \quad i = 0,1.
\end{displaymath}
For any $w \in H^q(\Omega_0 \cup \Omega_1)$, we define $\mc{R} w$ as
\begin{align*}
  (\mc{R} w )|_K := \begin{cases}
    (\mc{R}^0 (E^0 w))|_K, & \forall K \in \MTh^0 \backslash
    \MThG, \\
    (\mc{R}^1 (E^1 w))|_K, & \forall K \in \MTh^1 \backslash
    \MThG, \\
    (\mc{R}^i (E^i w))|_{K^i}, & \forall K \in 
    \MThG, \quad i = 0, 1. \\
  \end{cases}
\end{align*}
We denote by $U_h^m$ the image space of the operator $\mc{R}$, which 
is actually the reconstructed approximation space that will be used 
for solving the interface problem \eqref{eq_interface}. Actually, the
space $U_h^m$ is a combination of $U_h^{m,0}$ and $U_h^{m,1}$, which
can be represented as $U_h^m = U_h^{m,0} \cdot \chi_0 + U_h^{m,1}
\cdot \chi_1$, where $\chi_i$ is the characteristic function with 
respect to the domain $\Omega_i$. 
For the operator $\mc{R}$, we have the following local approximation 
error estimates: 
\begin{theorem}
  For any $K \in \MTh^i(i = 0, 1)$, there exists a constant $C$ such
  that 
  \begin{equation}
    \begin{aligned}
      \| E^i g - \mc{R} (E^i g) \|_{H^q(K)} &\leq C 
      \Lambda_m h_K^{m + 1 - q} \| E^i g \|_{H^{m+1}(S^i(K))},
      \quad 0 \leq q \leq m, \\
    \end{aligned}
    \label{eq_localapproximation2}
  \end{equation}
  for any $g \in H^{m+1}(\Omega_0 \cup \Omega_1)$.
  \label{th_localapproximation}
\end{theorem}
\begin{proof}
  The result directly follows from Lemma \ref{le_approximation}
  and the definition of the extension operator $E^i$.
\end{proof}
\begin{remark}
  In this section, we have constructed a discontinuous polynomial
  space $U_h^m$ based on the global reconstruction operator $\mc{R}$.
  Since each element has at most one unknown, we can
  construct an arbitrary order space without increasing the degrees 
  of freedom. This is in contrast to the normal DG space, which 
  requires a large number  of degrees of freedom on a single element 
  for achieving high-order accuracy \cite{Hughes2000comparison}. 
  From Theorem \ref{th_localapproximation}, the space $U_h^m$ has 
  almost the same approximation property as the normal DG space of 
  the same order as long as the constant $\Lambda_m$ is uniformly 
  bounded. Therefore, we can acquire high-order approximation with
  fewer degrees of freedom.
\end{remark}

\section{Approximation to Biharmonic Interface Problem}
\label{sec_interface_problem}
We define the approximation problem to solve the biharmonic interface
problem \eqref{eq_interface} based on the space $U_h^m$ constructed in
the previous section: \emph{Seek $u_h \in U_h^m$ ($m\geq 2$) such that} 
\begin{equation}
  B_h(u_h, v_h) = l_h(v_h), \quad \forall v_h \in U_h^m,
  \label{eq_variational}
\end{equation}
where the bilinear form $B_h(\cdot, \cdot)$ is defined for any $u,v
\in U_h := U_h^m + H^4(\Omega_0 \cup \Omega_1)$,
\begin{displaymath}
  \begin{aligned}
    B_h(u, v) &:= \sum_{K \in \MTh} \int_{K^0 \cup K^1} \beta \lap
    u \lap v \d{\bm{x}} + \sum_{e \in \MEh} \int_{e^0 \cup e^1} 
    \left( \jump{u} \cdot \aver{\nabla (\beta \lap v)} + 
    \jump{v} \cdot \aver{\nabla(\beta \lap u)} \right) \d{\bm{s}}
    \\
    &-\sum_{e \in \MEh} \int_{e^0 \cup e^1} \left( \jump{\nabla u}
    \aver{\beta \lap v} + \jump{\nabla v} \aver{\beta \lap u}
    \right) \d{\bm{s}} + \sum_{e \in \MEh} \int_{e^0 \cup e^1} \left(
    \mu_1 \jump{u} \cdot \jump{v} + \mu_2 \jump{\nabla u}
    \jump{\nabla v} \right) \d{\bm{s}} \\
    &+ \sum_{K \in \MThG} \int_{\Gamma_K} \left(
    \jump{u} \cdot \aver{\nabla(\beta \lap v)} + \jump{v} \cdot
    \aver{\nabla (\beta \lap u)} \right) \d{\bm{s}} \\
    &-\sum_{K \in \MThG} \int_{\Gamma_K} \left( \jump{\nabla u}
    \aver{\beta \lap v} + \jump{\nabla v} \aver{\beta \lap u}
    \right) \d{\bm{s}} + \sum_{K \in \MThG} \int_{\Gamma_K} 
    \left(\mu_1 \jump{u} \cdot \jump{v} + \mu_2 \jump{\nabla u} 
    \jump{\nabla v} \right) \d{\bm{s}}. \\
  \end{aligned}
\end{displaymath}
The parameter $\mu_1$ and $\mu_2$ are positive penalties which are set
by
\begin{displaymath}
  \begin{aligned}
    \mu_1|_e = \frac{\eta}{h_e^3}, \quad \mu_2|_e =
    \frac{\eta}{h_e}. \quad \text{ for } e \in \MEh, \\
    \mu_1|_{\Gamma_K} = \frac{\eta}{h_K^3}, \quad \mu_2|_{\Gamma_K} 
    = \frac{\eta}{h_K}. \quad \text{ for } K \in \MThG. 
  \end{aligned}
\end{displaymath}
The linear form $l_h(\cdot)$ is defined for $v \in U_h$,
\begin{displaymath}
  \begin{aligned}
    l_h(\cdot) &:= \sum_{K \in \MTh} \int_K f v \d{\bm{x}} + \sum_{K
    \in \MThG} \int_{\Gamma_K} a_1 \aver{\un_0 \cdot \nabla(\beta \lap
    v)} \d{\bm{s}} \\
    &-\sum_{K \in \MThG} \int_{\Gamma_K} a_2 \aver{\beta \lap v}
    \d{\bm{s}} + \sum_{K \in \MThG} \int_{\Gamma_K} a_3 \cdot
    \aver{\un_0 \cdot \nabla v} \d{\bm{s}} \\
    &-\sum_{K \in \MThG} \int_{\Gamma_K} a_4 \aver{v} \d{\bm{s}} +
    \sum_{K \in \MThG} \int_{\Gamma_K} \left( \mu_1 a_1 \un_0 \cdot
    \jump{v} + \mu_2 a_2 \jump{\nabla v} \right) \d{\bm{s}} \\
    &+ \sum_{e \in \MEh^b} \int_{e^0 \cup e^1} \left( g_1 (\jump{\nabla
    (\beta \lap v_h)} + \mu_1 \un \cdot \jump{v_h}) + g_2 (\un \cdot 
    \jump{-\beta \lap v_h} + \mu_2 \jump{\nabla v_h}) \right)
    \d{\bm{s}}.\\
  \end{aligned}
\end{displaymath}

We introduce two energy norms for the space $U_h$,
\begin{displaymath}
  \begin{aligned}
    \DGnorm{u}^2 &:= \sum_{K \in \MTh} \int_{K^0 \cup K^1} | \lap u
  |^2 \d{\bm{x}} + \sum_{e \in \MEh} \int_{e^0 \cup e^1} h_e^{-3} |
  \jump{u} |^2 \d{\bm{s}} 
  + \sum_{e \in \MEh} \int_{e^0 \cup e^1} h_e^{-1} | \jump{\nabla u}
  |^2 \d{\bm{s}} \\
  &+ \sum_{K \in \MThG} \int_{\Gamma_K} h_K^{-3} |
  \jump{u} |^2 \d{\bm{s}} 
  + \sum_{K \in \MThG} \int_{\Gamma_K} h_K^{-1} | \jump{\nabla u} |^2
  \d{\bm{s}},
  \end{aligned}
\end{displaymath}
and
\begin{displaymath}
  \begin{aligned}
    \DGenorm{u}^2 &:= \DGnorm{u}^2 + \sum_{e \in \MEh} \int_{e^0 \cup
    e^1} h_e^3 | \aver{\nabla \lap u} |^2 \d{\bm{s}} 
    + \sum_{e \in \MEh} \int_{e^0 \cup e^1} h_e | \aver{\lap u} |^2
    \d{\bm{s}} \\ 
    &+ \sum_{K \in \MThG} \int_{\Gamma_K} h_K | \aver{\lap
    u} |^2 \d{\bm{s}} 
    + \sum_{K \in \MThG} \int_{\Gamma_K} h_K^3 | \aver{\nabla \lap u}
    |^2 \d{\bm{s}}.
  \end{aligned}
\end{displaymath}
Before we start the error analysis, we first prove some trace
inequalities, which play a crucial role in the following analysis.
\begin{lemma}
  For any element $K \in \MThG$, there exists constant $C$ such that
  \begin{equation}
    \| \nabla^{\alpha} (\lap v_h) \|_{L^2(\partial K^i)} \leq C 
    h_K^{-1/2} \|\nabla^{\alpha} (\lap v_h) \|_{L^2(\circi{K})}, \quad 
    \forall v_h \in U_h^m, \quad i = 0,1, \quad \alpha = 0,1.
    \label{eq_traceBK}
  \end{equation}
  \label{le_traceBK}
\end{lemma}
\begin{proof}
  The proof of this result is quite similar to that in 
  \cite[Lemma 2]{Li2020interface} and so is omitted.
\end{proof}
\begin{lemma}
  For any element $K \in \MThG$, there exists a positive constant $h_0$
  independent of $h$ and the location of the interface such that 
  $\forall h \leq h_0$, 
  \begin{equation}
    \| v \|^2_{L^2(\Gamma_K)} \leq C \left( h_K^{-1}\| v \|^2_{L^2(K)} 
    + h_K \| \nabla v \|^2_{L^2(K)} \right), \quad \forall v 
    \in H^1(K).
    \label{eq_traceGK}
  \end{equation}
  \label{le_traceGK}
\end{lemma}
See the proof of this lemma in \cite{Hansbo2002unfittedFEM, 
Wu2012unfitted}.

We claim that the two energy norms are equivalent over the space 
$U_h^m$.
\begin{lemma}
  For any $u_h \in U_h^m$, there exists a constant $C$, such that
  \begin{equation}
    \DGnorm{u_h} \leq \DGenorm{u_h} \leq C \DGnorm{u_h}.
    \label{eq_normEq}
  \end{equation}
  \label{le_normEq}
\end{lemma}
\begin{proof}
  We only need to prove $\DGenorm{u_h} \leq C \DGnorm{u_h}$. For $e
  \in \MEh^\circ$, we denote the two neighbor elements of $e$ by $K^+$ and
  $K^-$. We have 
  \begin{displaymath}
    \| h_e^{3/2} \aver{\nabla \lap u_h} \|_{L^2(e^0 \cup e^1)}
    \leq C \sum_{i=0,1} \left(\| h_e^{3/2} \nabla \lap u_h 
    \|_{L^2(e^i \cap \partial K^+)} + \| h_e^{3/2} \nabla \lap u_h
    \|_{L^2(e^i \cap \partial K^-)} \right) 
  \end{displaymath}
  By the trace inequalities \eqref{eq_trace} and \eqref{eq_traceBK},
  and the inverse inequality \eqref{eq_inverse}, we obtain that
  \begin{displaymath}
    \|h_e^{3/2} \nabla \lap u_h\|_{L^2(e^i \cap \partial K^{\pm})}
    \leq \left\{ 
    \begin{aligned}
      &C \| \lap u_h \|_{L^2(K^{\pm})}, \quad \text{ if } K^{\pm} \in
      \MThB, \\
      &C \| \lap u_h \|_{L^2( \circi{(K^{\pm})})}, \quad \text{ if } 
      K^{\pm} \in \MThG.
    \end{aligned}
    \right. \quad i = 0,1.
  \end{displaymath}
  For $e \in \MEh^b$, let $e \subset K$. Similarly, we have
  \begin{displaymath}
    \|h_e^{3/2} \nabla \lap u_h\|_{L^2(e^i \cap \partial K)}
    \leq \left\{ 
    \begin{aligned}
      &C \| \lap u_h \|_{L^2(K)}, \quad \text{ if } K \in
      \MThB, \\
      &C \| \lap u_h \|_{L^2( \circi{K})}, \quad \text{ if } 
      K \in \MThG.
    \end{aligned}
    \right. \quad i = 0,1.
  \end{displaymath}
  For $K \in \MThG$, using the inequality \eqref{eq_traceGK}, we have 
  \begin{displaymath}
    \begin{aligned}
      \|h_K^{3/2} \aver{\nabla \lap u_h} \|_{L^2(\Gamma_K)} &\leq C 
      \left( \|h_K^{3/2} \nabla \lap u_h\|_{L^2(\partial K^0)} + 
      \| h_K^{3/2} \nabla \lap u_h \|_{L^2(\partial K^1)} \right) \\
      &\leq C \left( \| \lap u_h \|_{L^2(K_\circ^0)} + \| \lap u_h 
      \|_{L^2(K_\circ^1)} \right).
    \end{aligned}
  \end{displaymath}
  The terms $\| h_e^{1/2} \aver{\lap u_h} \|_{L^2(e^0 \cup e^1)}$ and 
  $\| h_K^{1/2} \aver{\lap u_h} \|_{L^2(\Gamma_K)}$ can bounded by the 
  same way. Thus, by summing over all $e \in \MEh$ and $K \in \MThG$,
  we conclude that
  \begin{displaymath}
    \begin{aligned}
    \sum_{e \in \MEh} \| h_e^{1/2} \aver{\lap u_h} \|_{L^2(e^0 \cup
    e^1)}^2 + \sum_{e \in \MEh} \| h_e^{3/2} \aver{\nabla \lap u_h}
    \|_{L^2(e^0 \cup e^1)}^2 &+ \sum_{K \in \MThG} \| h_K^{1/2} 
    \aver{\lap u_h} \|_{L^2(\Gamma_K)}^2 \\
    &+ \sum_{K \in \MThG} 
    \| h_K^{3/2} \aver{\nabla \lap u_h} \|_{L^2(\Gamma_K)}^2 \\
    &\leq C \sum_{K \in \MTh} \| \lap u_h \|_{L^2(K)}^2,
    \end{aligned}
  \end{displaymath}
  which can immediately leads us to \eqref{eq_normEq} and completes
  the proof.
\end{proof}
Now we are ready to prove the coercivity and the continuity of the
bilinear form $B_h(\cdot, \cdot)$.
\begin{theorem}
  Let $B_h(\cdot, \cdot)$ be the bilinear form with sufficiently large
  penalty $\eta$. Then there exists a positive constant $C$ such that 
  \begin{equation}
    B_h(u_h, u_h) \geq C \DGenorm{u_h}^2, \quad \forall u_h \in U_h^m.
    \label{eq_coercivity}
  \end{equation}
  \label{th_coercivity}
\end{theorem}
\begin{proof}
  From Lemma \ref{le_normEq}, we only need to establish the coercivity
  over the norm $\DGnorm{\cdot}$. For the face $e \in \MEh^\circ$, let
  $e$ be shared by the neighbor elements $K^-$ and $K^+$. We apply the
  Cauchy-Schwarz inequality to get that
  \begin{displaymath}
    \begin{aligned}
      -\int_{e^0 \cup e^1} 2\jump{u_h} \cdot &\aver{\nabla (\beta \lap
      u_h)} \d{\bm{s}} \geq -\frac{1}{\epsilon}\|h_e^{-3/2}
      \jump{u_h}\|^2_{L^2(e^0 \cup e^1)} - \epsilon \|h_e^{3/2} 
      \aver{\nabla(\beta \lap u_h)} \|^2_{L^2(e^0 \cup e^1)} \\
      &\geq \sum_{i=0,1} C\left( -\frac{1}{\epsilon}\|h_e^{-3/2}
      \jump{u_h} \|^2_{L^2(e^i)} -\epsilon \| h_e^{3/2} \nabla(\beta
      \lap u_h) \|^2_{L^2(e^i \cap \partial K^-)} - \epsilon \|
      h_e^{3/2} \nabla (\beta \lap u_h) \|^2_{L^2(e^i \cap \partial
      K^+)} \right),
    \end{aligned}
  \end{displaymath}
  for any $\epsilon > 0$. From the trace inequalities
  \eqref{eq_trace}, \eqref{eq_traceBK} and the inverse inequality
  \eqref{eq_inverse}, we deduce that
  \begin{displaymath}
    \|h_e^{3/2} \nabla (\beta \lap u_h)\|_{L^2(e^i \cap \partial K^{\pm})}
    \leq \left\{ 
    \begin{aligned}
      &C \| \lap u_h \|_{L^2(K^{\pm})}, \quad \text{ if } K^{\pm} \in
      \MThB, \\
      &C \| \lap u_h \|_{L^2( \circi{(K^{\pm})})}, \quad \text{ if } 
      K^{\pm} \in \MThG.
    \end{aligned}
    \right. \quad i = 0,1.
  \end{displaymath}
  Thus, we have
  \begin{equation}
    -\sum_{e \in \MEh^\circ} \int_{e^0 \cup e^1} 2 \jump{u_h} \cdot
    \aver{\nabla(\beta \lap u_h)} \d{\bm{s}} \geq -\sum_{e \in
    \MEh^\circ} \frac{1}{\epsilon} \| h_e^{-3/2} \jump{u_h}
    \|^2_{L^2(e^0 \cup e^1)} - C\epsilon \sum_{K \in \MTh} \| \lap u_h
    \|^2_{L^2(K^0 \cup K^1)}.
    \label{eq_eo1}
  \end{equation}
  For the face $e \in \MEh^b$, we can similarly derive that
  \begin{equation}
    -\sum_{e \in \MEh^b} \int_{e^0 \cup e^1} 2 \jump{u_h} \cdot
    \aver{\nabla(\beta \lap u_h)} \d{\bm{s}} \geq -\sum_{e \in
    \MEh^b} \frac{1}{\epsilon} \| h_e^{-3/2} \jump{u_h}
    \|^2_{L^2(e^0 \cup e^1)} - C\epsilon \sum_{K \in \MTh} \| \lap u_h
    \|^2_{L^2(K^0 \cup K^1)}.
    \label{eq_eb1}
  \end{equation}
  For the element $K \in \MThG$, we apply the Cauchy-Schwarz
  inequality, the trace estimate \eqref{eq_traceBK} and the inverse
  inequality \eqref{eq_inverse} to get that
  \begin{displaymath}
    \begin{aligned}
      -\int_{\Gamma_K} 2\jump{u_h} \cdot \aver{\nabla(\beta \lap u_h)}
      \d{\bm{s}} &\geq -\frac{1}{\epsilon} \|h_K^{-3/2} \jump{u_h}
      \|^2_{L^2(\Gamma_K)} - C\epsilon \sum_{i=0,1} \| h_K^{3/2}
      \nabla(\beta \lap u_h) \|^2_{L^2(\Gamma_K \cap K^i)} \\
      &\geq -\frac{1}{\epsilon} \|h_K^{-3/2} \jump{u_h}
      \|^2_{L^2(\Gamma)} - C \epsilon \sum_{i=0,1} \| \lap u_h
      \|^2_{L^2(\circi{K})},
    \end{aligned}
  \end{displaymath}
  which implies that
  \begin{equation}
    -\sum_{K \in \MThG} \int_{\Gamma_K} 2 \jump{u_h} \cdot
    \aver{\nabla(\beta \lap u_h)} \d{\bm{s}} \geq -\sum_{K \in \MThG}
    \frac{1}{\epsilon} \| h_K^{-3/2} \jump{u_h} \|^2_{L^2(\Gamma_K)} -
    C\epsilon \sum_{K \in \MTh} \|\lap u_h \|^2_{L^2(K^0 \cup K^1)}.
    \label{eq_GK0}
  \end{equation}
  By employing the same method to the term $\int_{e^0 \cup e^1} 2 
  \jump{\nabla u_h} \aver{\beta \lap u_h} \d{\bm{s}}$ and 
  $\int_{\Gamma_K} 2 \jump{\nabla u_h} \aver{\beta \lap u_h} 
  \d{\bm{s}}$, we can obtain that
  \begin{equation}
    -\sum_{e \in \MEh} \int_{e^0 \cup e^1} 2 \jump{\nabla u_h}
    \aver{\beta \lap u_h} \d{\bm{s}} \geq -\frac{1}{\epsilon} \sum_{e
    \in \MEh} \| h_e^{-1/2} \jump{\nabla u_h} \|^2_{L^2(e^0 \cup e^1)}
    - C\epsilon \sum_{K \in \MTh} \| \lap u_h \|^2_{L^2(K^0 \cup
    K^1)},
    \label{eq_eob2}
  \end{equation}
  and
  \begin{equation}
    -\sum_{K \in \MThG} \int_{\Gamma_K} 2 \jump{\nabla u_h}
    \aver{\beta \lap u_h} \d{\bm{s}} \geq -\frac{1}{\epsilon} \sum_{K
    \in \MThG} \| h_K^{-1/2} \jump{\nabla u_h} \|^2_{L^2(\Gamma_K)}
    - C\epsilon \sum_{K \in \MTh} \| \lap u_h \|^2_{L^2(K^0 \cup
    K^1)}.
    \label{eq_GK1}
  \end{equation}
  Combining the inequalities \eqref{eq_eo1}, \eqref{eq_eb1},
  \eqref{eq_eob2}, \eqref{eq_GK0} and \eqref{eq_GK1}, we conclude that
  there exists a constant $C$ such that
  \begin{displaymath}
    \begin{aligned}
      B_h(u_h, u_h) &\geq (\beta_{\min} - C\epsilon) 
      \sum_{K \in \MTh} \| \lap u_h \|^2_{L^2(K^0 \cup K^1)} \\
    &+ (\eta - \frac{1}{\epsilon}) \sum_{e
    \in \MEh} (\| h_e^{-1/2} \jump{\nabla u_h} \|^2_{L^2(e^0 \cup
    e^1)} + \| h_e^{-3/2} \jump{u_h} \|^2_{L^2(e^0 \cup e^1)}) \\
    &+ (\eta - \frac{1}{\epsilon}) \sum_{K \in \MThG} 
    (\| h_K^{-1/2} \jump{\nabla u_h} \|^2_{L^2(\Gamma_K)} + 
    \| h_K^{-3/2} \jump{u_h} \|^2_{L^2(\Gamma_K)}),
    \end{aligned}
  \end{displaymath}
  for any $\epsilon > 0$, where $\beta_{\min} > 0$ is the minimum 
  value of $\beta$. We can let $\epsilon = \beta_{\min}/(2C)$ and
  select a sufficiently large $\eta$ to ensure $B_h(u_h, u_h) \geq C
  \DGnorm{u_h}^2$, which completes the proof.
\end{proof}
\begin{theorem}
  There exists a positive constant $C$ such that
  \begin{equation}
    |B_h(u,v)| \leq C \DGenorm{u} \DGenorm{v}, \quad \forall u,v \in
    U_h.
    \label{eq_continuity}
  \end{equation}
  \label{th_continuity}
\end{theorem}
\begin{proof}
  The continuity can be proved by directly using the Cauchy-Schwarz
  inequality,
  \begin{displaymath}
    \begin{aligned}
      B_h(u,v) \leq C &\left( \sum_{K \in \MTh} \| \beta \lap u 
      \|_{L^2(K^0 \cup K^1)}^2 + \sum_{e \in \MEh} (\| h_e^{-3/2} 
      \jump{u} \|^2_{L^2(e^0 \cup e^1)} + \| h_e^{-1/2} 
      \jump{ \nabla u } \|^2_{L^2(e^0 \cup e^1)} \right. \\
      &+ \| h_e^{3/2} \aver{\nabla(\beta \lap
      u)} \|^2_{L^2(e^0 \cup e^1)} + \| h_e^{1/2} \aver{\beta \lap u}
      \|^2_{L^2(e^0 \cup e^1)}) 
      + \sum_{K \in \MThG} (\| h_K^{-3/2} \jump{u}
      \|^2_{L^2(\Gamma_K)} \\
      & \left. + \| h_K^{-1/2} \jump{ \nabla u }
      \|^2_{L^2(\Gamma_K)} + \| h_K^{3/2} \aver{\nabla(\beta \lap
      u)} \|^2_{L^2(\Gamma_K)} + \| h_K^{1/2} \aver{\beta \lap u}
      \|^2_{L^2(\Gamma_K)}) \right)^{1/2} \cdot \\
      &\left( \sum_{K \in \MTh} \| \beta \lap v 
      \|_{L^2(K^0 \cup K^1)}^2 + \sum_{e \in \MEh} (\| h_e^{-3/2} 
      \jump{v} \|^2_{L^2(e^0 \cup e^1)} + \| h_e^{-1/2} 
      \jump{ \nabla v } \|^2_{L^2(e^0 \cup e^1)} \right. \\
      &+ \| h_e^{3/2} \aver{\nabla(\beta \lap
      v)} \|^2_{L^2(e^0 \cup e^1)} + \| h_e^{1/2} \aver{\beta \lap v}
      \|^2_{L^2(e^0 \cup e^1)}) 
      + \sum_{K \in \MThG} (\| h_K^{-3/2} \jump{v}
      \|^2_{L^2(\Gamma_K)} \\
      & \left. + \| h_K^{-1/2} \jump{ \nabla v }
      \|^2_{L^2(\Gamma_K)} + \| h_K^{3/2} \aver{\nabla(\beta \lap
      v)} \|^2_{L^2(\Gamma_K)} + \| h_K^{1/2} \aver{\beta \lap v}
      \|^2_{L^2(\Gamma_K)}) \right)^{1/2},
    \end{aligned}
  \end{displaymath}
  which completes the proof.
\end{proof}
Now we verify the Galerkin orthogonality to the bilinear form
$B_h(\cdot,\cdot)$.
\begin{lemma}
  Suppose $u \in H^4(\Omega_0 \cup \Omega_1)$ is the exact solution to
  the problem \eqref{eq_interface}, and $u_h \in U_h^m$ is the
  numerical solution to the discrete problem \eqref{eq_variational},
  then 
  \begin{equation}
    B_h(u - u_h, v_h) = 0, \quad \forall v_h \in U_h^m.
    \label{eq_orthogonality}
  \end{equation}
  \label{le_orthogonality}
\end{lemma}
\begin{proof}
  Since $u \in H^4(\Omega_0 \cup \Omega_1)$, we have that
  \begin{displaymath}
    \jump{u}|_{e^i} = \bm{0}, \quad \jump{\nabla u}|_{e^i} = 0, \quad
    \jump{\beta \lap u}|_{e^i} = \bm{0}, \quad \jump{\nabla(\beta \lap
    u)}|_{e^i} = 0, \quad \forall e \in \MEh^\circ, \,\, i = 0,1.
  \end{displaymath}
Taking the exact solution into $B_h(\cdot, \cdot)$, we have that
\begin{displaymath}
  \begin{aligned}
    B_h(u, v_h) &= \sum_{K \in \MTh} \int_{K^0 \cup K^1} \beta \lap
    u \lap v_h \d{\bm{x}} + \sum_{e \in \MEh} \int_{e^0 \cup e^1}  
    \jump{v_h} \cdot \aver{\nabla(\beta \lap u)} \d{\bm{s}}
    \\
    &-\sum_{e \in \MEh} \int_{e^0 \cup e^1} \jump{\nabla v_h} 
    \aver{\beta \lap u}  \d{\bm{s}} + \sum_{K \in \MThG} 
    \int_{\Gamma_K} \jump{v_h} \cdot \aver{\nabla (\beta \lap u)} 
    \d{\bm{s}} \\
    &-\sum_{K \in \MThG} \int_{\Gamma_K} \jump{\nabla v_h} 
    \aver{\beta \lap u} \d{\bm{s}} + \sum_{K \in \MThG} \int_{\Gamma_K} 
    \left( a_1 \aver{\un_0 \cdot \nabla(\beta \lap v_h)} - a_2
    \aver{\beta \lap v_h} \right) \d{\bm{s}} \\
    &+ \sum_{e \in \MEh^b} \int_{e^0 \cup e^1} \left( g_1 (\jump{\nabla
    (\beta \lap v_h)} + \mu_1 \un \cdot \jump{v_h}) + g_2 (\un \cdot 
    \jump{-\beta \lap v_h} + \mu_2 \jump{\nabla v_h}) \right) \d{\bm{s}}\\
  \end{aligned}
\end{displaymath}
We multiply the test function $v_h$ at both side of equation
\eqref{eq_interface}, and apply the integration by parts to get that
\begin{displaymath}
  \begin{aligned}
    \sum_{K \in \MTh} \int_{K^0 \cup K^1} &f v_h \d{\bm{x}} = 
    \sum_{K \in \MTh} \int_{K^0 \cup K^1} \lap(\beta
    \lap u) v_h \d{\bm{x}} = \sum_{K \in \MTh} \int_{K^0 \cup K^1}
    \beta \lap u \lap v_h \d{\bm{x}} \\
    &+ \sum_{K \in \MTh} \sum_{i=0,1} \left( \int_{\partial K^i} \un
    \cdot \nabla (\beta \lap u) v_h \d{\bm{s}} - \int_{\partial K^i}
    \beta \lap u \un \cdot \nabla v_h \d{\bm{s}} \right) \\
    &=\sum_{K \in \MTh} \int_{K^0 \cup K^1} \beta \lap u \lap v_h 
    \d{\bm{x}} + \sum_{e \in \MEh} \int_{e^0 \cup e^1}  
    \jump{v_h} \cdot \aver{\nabla(\beta \lap u)} \d{\bm{s}}
    \\
    &-\sum_{e \in \MEh} \int_{e^0 \cup e^1} \jump{\nabla v_h} 
    \aver{\beta \lap u}  \d{\bm{s}} + \sum_{K \in \MThG} 
    \int_{\Gamma_K} \jump{v_h} \cdot \aver{\nabla (\beta \lap u)} 
    \d{\bm{s}} \\
    &-\sum_{K \in \MThG} \int_{\Gamma_K} \jump{\nabla v_h} 
    \aver{\beta \lap u} \d{\bm{s}} + \sum_{K \in \MThG} \int_{\Gamma_K} 
    a_4 \aver{v_h} \d{\bm{s}}
    -\sum_{K \in \MThG} \int_{\Gamma_K} a_3 \aver{\un_0 \cdot \nabla
    v_h} \d{\bm{s}}.
  \end{aligned}
\end{displaymath}
Thus, by simply calculating, we obtain that
\begin{displaymath}
  B_h(u_h, v_h) = l_h(v_h) = B_h(u, v_h),
\end{displaymath}
which completes the proof.
\end{proof}
Then we establish the interpolation error estimate to $\mc{R}$ under
the norm $\DGenorm{\cdot}$.
\begin{lemma}
  For $0 \leq h \leq h_0$ and $m \geq 2$, there exists a constant $C$ 
  such that
  \begin{equation}
    \DGenorm{v - \mc{R} v} \leq C \Lambda_m h^{m-1} \| v
    \|_{H^{m+1}(\Omega_0 \cup \Omega_1)}, \quad \forall v \in
    H^{s}(\Omega_0 \cup \Omega_1), \quad s = \max(4, m+1).
    \label{eq_interpolation_err}
  \end{equation}
  \label{le_interpolation_err}
\end{lemma}
\begin{proof}
  From the definition of the extension operator $E^i$ and Theorem
  \ref{th_localapproximation}, we can show that
  \begin{displaymath}
    \begin{aligned}
      \sum_{K \in \MTh^i} \| \lap v - \lap(\mc{R} v) \|^2_{L^2(K^i)} 
      &\leq \sum_{K \in \MTh^i} \| \lap(E^i v) - \lap(\mc{R}(E^i v) 
      \|^2_{L^2(K)} \\
      &\leq \sum_{K \in \MTh^i} C \Lambda_m^2 h_K^{2m-2} \| E^i v
      \|^2_{H^{m+1}(S^i(K))} \\
      &\leq C \Lambda_m^2 h^{2m-2} \| E^i v \|^2_{H^{m+1}(\Omega)} \leq
      C \Lambda_m^2 h^{2m-2} \| v \|^2_{H^{m+1}(\Omega_i)}, \\
    \end{aligned}
  \end{displaymath}
  for $i = 0, 1$. By the trace estimate \eqref{eq_traceGK}, 
  \begin{displaymath}
    \begin{aligned}
      \sum_{K \in \MThG}  h_K^{-1} \| \jump{\nabla ( v - \mc{R} v)} 
      \|^2_{L^2(\Gamma_K)} &\leq C \sum_{K \in \MThG} \sum_{i=0,1}
      \left( h_K^{-2} \| \nabla(E^i v - \mc{R} (E^i v)) \|^2_{L^2(K)}
      + \| \nabla (E^i v - \mc{R}(E^i v)) \|^2_{H^1(K)} \right) \\
      &\leq C \Lambda_m^2 h^{2m-2} \sum_{i=0,1} \| E^i v
      \|^2_{H^{m+1}(\Omega)} \leq C \Lambda_m^2 h^{2m-2} \| v
      \|^2_{H^{m+1}(\Omega_0 \cup \Omega_1)}.
    \end{aligned}
  \end{displaymath}
  The other trace terms in $\DGenorm{v - \mc{R} v}$ can be bounded by
  the trace estimates \eqref{eq_traceGK} and \eqref{eq_trace}
  similarly, which completes the proof.
\end{proof}
Now we are ready to present the \emph{a priori} error estimate under
$\DGenorm{\cdot}$ within the standard Lax-Milgram framework.
\begin{theorem}
  Suppose the biharmonic interface problem \eqref{eq_interface} has a
  solution $u \in H^{s}(\Omega_0 \cup \Omega_1)$, where $s = \max(4,
  m+1)$, $m \geq 2$ and $\Lambda_m$ has a uniform upper bound 
  independent of $h$.
  Let the bilinear form $B_h(\cdot, \cdot)$ be defined with a
  sufficiently large $\eta$ and $u_h in U_h^m$ be the numerical 
  solution to the problem \eqref{eq_variational}. Then for $h
  \leq h_0$ there exists a constant $C$ such that
  \begin{equation}
    \DGenorm{u - u_h} \leq C h^{m-1} \| u \|_{H^{m+1}(\Omega_0
    \cup \Omega_1)}.
    \label{eq_errorestimate}
  \end{equation}
  \label{th_errorestimate}
\end{theorem}
\begin{proof}
  From \eqref{eq_coercivity}, \eqref{eq_continuity} and
  \eqref{eq_orthogonality}, we have that for any $v_h \in U_h^m$,
  \begin{displaymath}
    \begin{aligned}
      \DGenorm{u_h - v_h}^2 &\leq C B_h(u_h - v_h, u_h - v_h) = C
      B_h(u - u_h, u_h - v_h) \\
      &\leq C \DGenorm{u - v_h} \DGenorm{u_h - v_h}.
    \end{aligned}
  \end{displaymath}
  By the triangle inequality, there holds
  \begin{displaymath}
    \DGenorm{u - u_h} \leq \DGenorm{u - v_h} + \DGenorm{v_h - u_h}
    \leq C \inf_{v_h \in U_h^m} \DGenorm{u - v_h}.
  \end{displaymath}
  Let $v_h = \mc{R} u$, by the inequality
  \eqref{eq_interpolation_err}, we arrive at 
  \begin{displaymath}
    \DGenorm{u - u_h} \leq C \DGenorm{u - \mc{R} u} \leq C h^{m-1}
    \|u\|_{H^{m+1}(\Omega_0 \cup \Omega_1)},
  \end{displaymath}
  which completes the proof.
\end{proof}
Ultimately, we prove the $L^2$ error estimate by the duality argument.
Let $\phi \in H^4(\Omega_0 \cup \Omega_1)$ be the solution of the problem
\begin{equation}
  \left\{
  \begin{aligned}
    &\lap (\beta(x) \lap \phi) = u - u_h, \quad \text{ in } \Omega_0
    \cup \Omega_1, \\
    &\phi = \parn{\phi} = 0, \quad \text{ on } \partial \Omega, \\
    & \jump{\phi} = \jump{\beta \lap \phi} = \bm{0}, \quad \text{ on }
    \Gamma, \\
    & \jump{\nabla \phi} = \jump{\nabla(\beta \lap \phi)} = 0, \quad 
    \text{ on } \Gamma,
  \end{aligned}
  \right.
  \label{eq_dualinterface}
\end{equation}
and satisfies 
\begin{equation}
  \| \phi \|_{H^4(\Omega_0 \cup \Omega_1)} \leq C \| u - u_h
  \|_{L^2(\Omega)}.
  \label{eq_dualregularity}
\end{equation}
We refer to \cite{mozolevski2003priori} for the details of the
regularity assumption. 

From the Galerkin orthogonality \eqref{eq_orthogonality} and the
Theorem \ref{th_errorestimate}, we deduce that
\begin{displaymath}
  \begin{aligned}
    \| u - u_h \|^2_{L^2(\Omega)} &= B_h( \phi, u-u_h) = B_h(\phi -
    \mc{R} \phi, u- u_h) \\
    &\leq C \DGenorm{\phi - \mc{R} \phi} \DGenorm{u - u_h} \\
    &\leq 
    \begin{cases}
      &C h^2 \| u \|_{H^3(\Omega_0 \cup \Omega_1)}, \quad m =2, \\
      &C h^{m+1} \| u \|_{H^{m+1}(\Omega_0 \cup \Omega_1)}, \quad m
      \geq 3.
    \end{cases}
  \end{aligned}
\end{displaymath}
We summarize what we have proved as the following theorem
\begin{theorem}
  Suppose the conditions in Theorem \ref{th_errorestimate} and the
  assumption \eqref{eq_dualregularity} hold true, then we have
  \begin{equation}
    \begin{aligned}
      \| u - u_h \|_{L^2(\Omega)} &\leq C h^2 \| u \|_{H^3(\Omega_0
      \cup \Omega_1)}, \quad m = 2, \\
      \| u - u_h \|_{L^2(\Omega)} &\leq C h^{m+1} \| u \|_{H^{m+1}
      (\Omega_0 \cup \Omega_1)}, \quad m \geq 3, 
    \end{aligned}
    \label{eq_L2errorestimate}
  \end{equation}
  \label{th_L2errorestimate}
\end{theorem}

\section{Numerical Results}
\label{sec_numericalresults}
In this section, we conduct a series numerical experiments to test the
performance of our method. For the accuracy $2 \leq m \leq 6$, the
penalty parameter and the threshold $\# S$ we used in Example 1, 2, 4
and 5 are listed in Tab.~\ref{tab_pen_and_patch}. For all examples, 
the jump conditions
$a_1, a_2, a_3, a_4$, the boundary $g_1, g_2$ and the right hand side
$f$ in the equation \eqref{eq_interface} are chosen according to the
exact solution. The integrals on the interface and the curved domains
like
\begin{displaymath}
  \int_{K^0} h(\bm{x}) \d{\bm{x}}, \quad \int_{K^1} h_(\bm{x})
  \d{\bm{x}}, \quad \int_{\Gamma_K} h(\bm{x}) \d{\bm{s}}, \quad
  \forall K \in \MThG,
\end{displaymath}
in the numerical scheme are implemented by the PHG package
\cite{Zhang2009parallel}. 

\noindent \textbf{Example 1.} 
We first consider a problem defined on the squared domain $\Omega =
(-1,1)^2$ with a circular interface inside it. The interface is
described by the level function
\begin{equation}
  \phi(x,y) = x^2 + y^2 - r^2, \quad r = 0.5,
  \label{eq_ex1interface}
\end{equation}
see Fig.~\ref{fig_ex1_interface}. The exact solution is chosen by
\begin{displaymath}
  u(x,y) = \left\{
  \begin{aligned}
    &e^{x^2 + y^2}, \quad \text{ in } \Omega_0,\\
    &0.1(x^2 + y^2)^2 -0.005\mathop{ln}(x^2 + y^2), \quad \text{ in } \Omega_1,
  \end{aligned}
  \right.
\end{displaymath}
with the coefficient to be
\begin{displaymath}
  \beta = \left\{
  \begin{aligned}
    &1, \quad \text{ in } \Omega_0, \\
    &10, \quad \text{ in } \Omega_1.
  \end{aligned}
  \right.
\end{displaymath}
We solve the interface problem on a sequence of meshes with the size
$h = 1/10, 1/20, 1/40, 1/80$. The convergence histories under 
the $\DGenorm{\cdot}$ and $\| \cdot \|_{L^2(\Omega)}$ are shown in
Fig.~\ref{fig_ex1err}. The error under the energy norm is decreasing 
at the speed $O(h^{m-1})$ for fixed $m$. For $L^2$ error, the speed is
$O(h^2)$ when $m = 2$ and $O(h^{m+1})$ when $m \geq 3$. This results
are coincide with the theoretical analysis in Theorem
\ref{th_errorestimate} and Theorem \ref{th_L2errorestimate}.

\begin{table}[htp]
\begin{minipage}[t]{0.3\textwidth}
  \centering
  \begin{tabular}{p{0.6cm}|p{0.6cm}|p{0.6cm}|p{0.6cm}|p{0.6cm}|p{0.6cm}}
  \hline\hline
  $m$    & 2 & 3 & 4 & 5 & 6 \\ \hline
  $\eta$ & 20 & 20 & 20 & 35 & 35 \\ \hline
  $\# S$ & 12 & 18 & 25 & 32 & 55 \\ 
  \hline\hline
  \end{tabular}
\end{minipage}
\hspace{2cm}
\begin{minipage}[t]{0.3\textwidth}
  \begin{tabular}{p{0.6cm}|p{0.6cm}|p{0.6cm}}
    \hline\hline
   $m$ & 2 & 3  \\ \hline
   $\eta$ & 35 & 50 \\ \hline
   $\# S$ & 25 & 45  \\ 
    \hline\hline
  \end{tabular}
\end{minipage}
\caption{The $\eta$ and $\# S$ used in 2D and 3D examples.} 
\label{tab_pen_and_patch}
\end{table}

\begin{figure}[htp]
  \centering
  \begin{minipage}[t]{0.46\textwidth}
    \centering
    \begin{tikzpicture}[scale=2]
      \centering
      \draw[thick, black] (-1, -1) rectangle (1, 1);
      \draw[thick,red] (0,0) circle [radius = 0.5];
    \end{tikzpicture}
  \end{minipage}
  \begin{minipage}[t]{0.46\textwidth}
    \centering
    \begin{tikzpicture}[scale=2]
      \centering
      \input{./figure/mth.tex}
      \draw[thick,red] (0,0) circle [radius = 0.5];
    \end{tikzpicture}
  \end{minipage}
  \caption{The interface and the mesh in Example 1.}
  \label{fig_ex1_interface}
\end{figure}

\begin{figure}[htbp]
  \centering
  \includegraphics[width=0.30\textwidth]{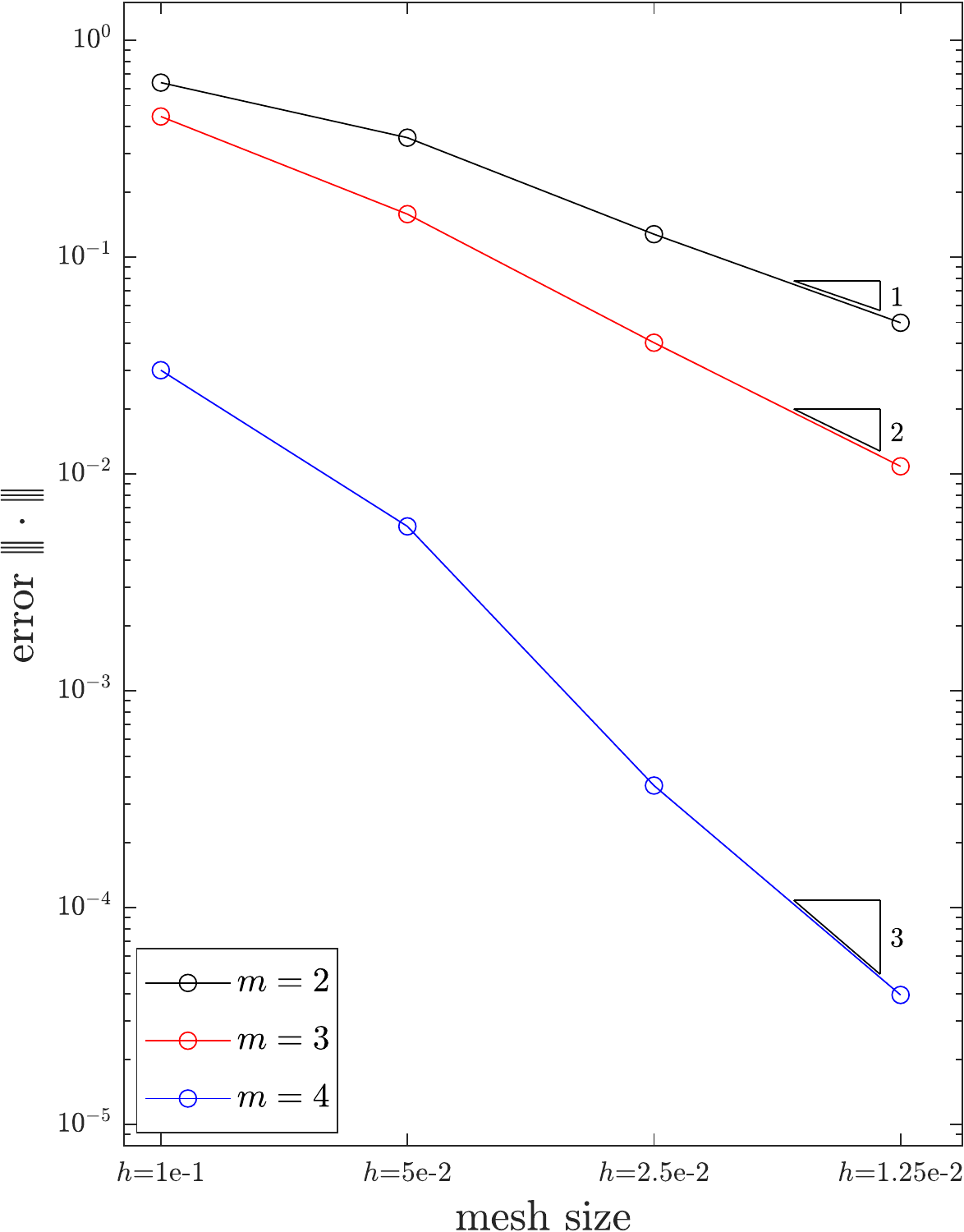}
  \hspace{30pt}
  \includegraphics[width=0.30\textwidth]{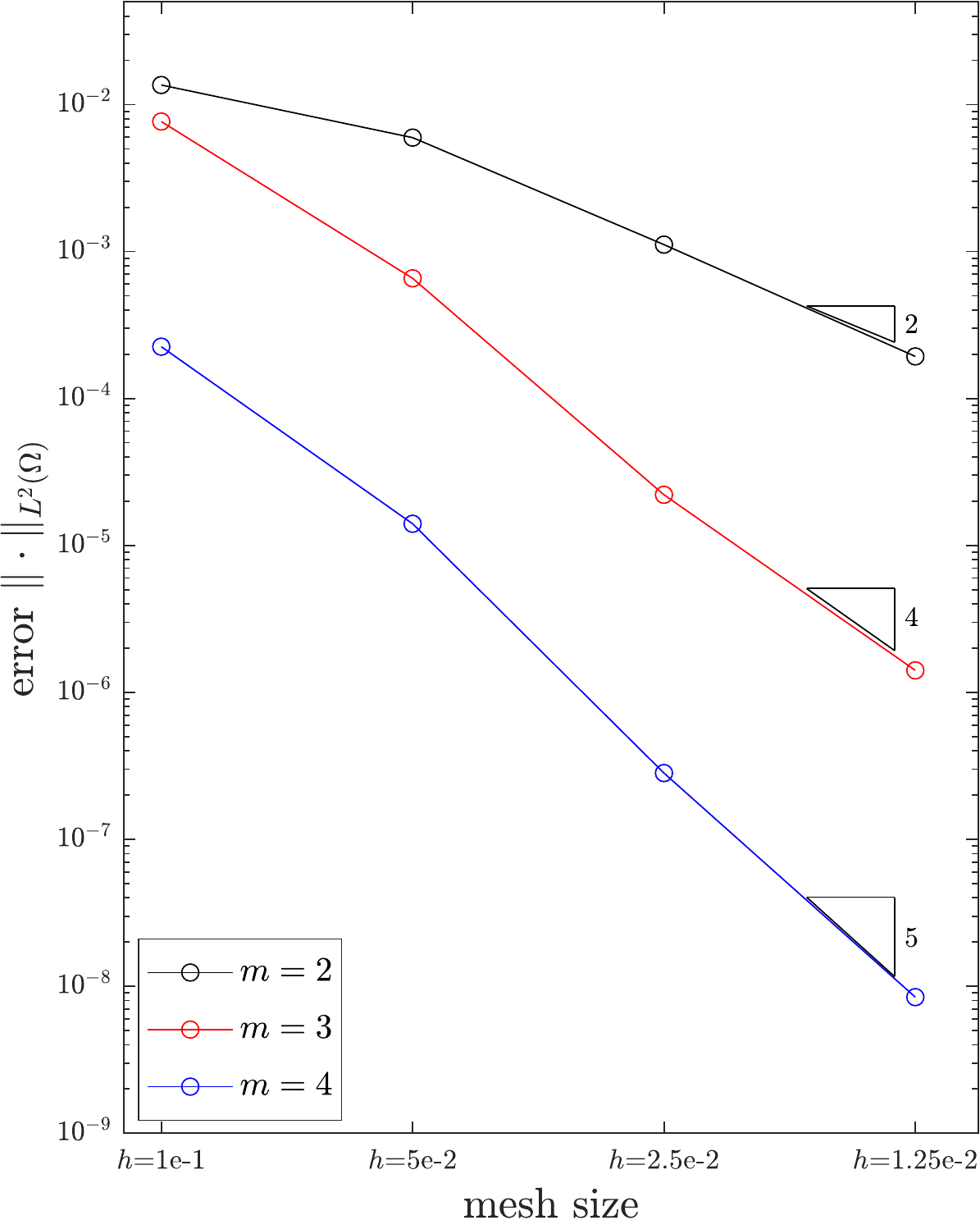}
  \caption{The convergence histories under the $\DGenorm{\cdot}$
  (left) and the $\| \cdot \|_{L^2(\Omega)}$ (right) in Example 1.}
  \label{fig_ex1err}
\end{figure}

\noindent \textbf{Example 2.}
We test the high-order approximation property in this example by 
solving a biharmonic interface problem on the squared domain $\Omega =
(-1,1)^2$. The interface is still given by \eqref{eq_ex1interface}.
The exact solution is taken as
\begin{displaymath}
  u(x,y) = \left\{
  \begin{aligned}
    &\sin^2(2x) \sin^2(2y), \quad \text{ in } \Omega_0,\\
    &\sin(2x) \sin(2y), \quad \text{ in } \Omega_1,
  \end{aligned}
  \right.
\end{displaymath}
The reconstruction order $m$ ranges from $2$ to $6$. And the $\Omega$
is partitioned into triangle mesh with size $h = 0.15, 0.075, 0.0375,
0.01875$. We display the numerical results in Fig.~\ref{fig_ex2err}. 
It can be seen that all the convergence rates are consistent with our
theoretical results and our method can achieve high-order accuracy by
applying high-order reconstruction.
\begin{figure}[htbp]
  \centering
  \includegraphics[width=0.30\textwidth]{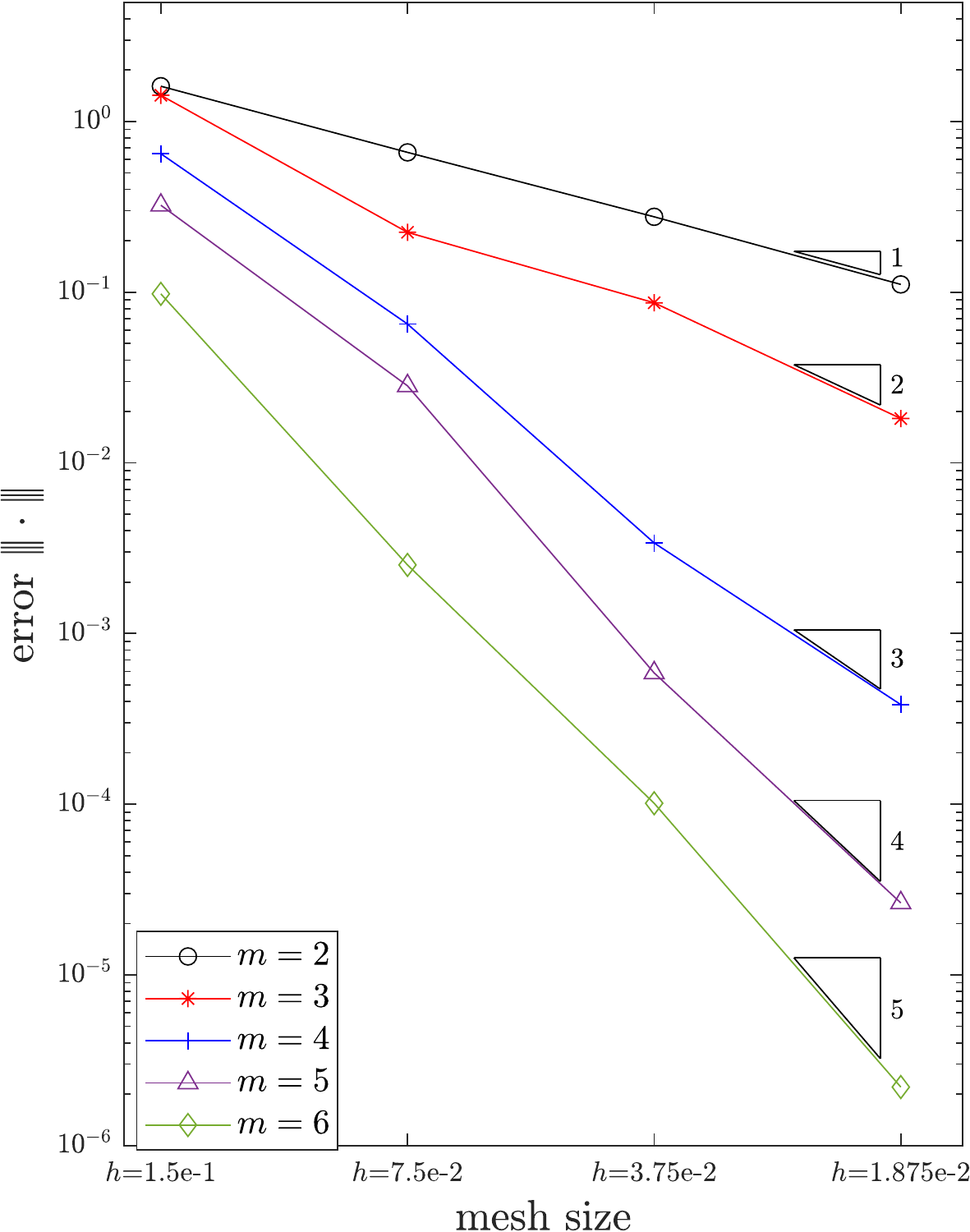}
  \hspace{30pt}
  \includegraphics[width=0.30\textwidth]{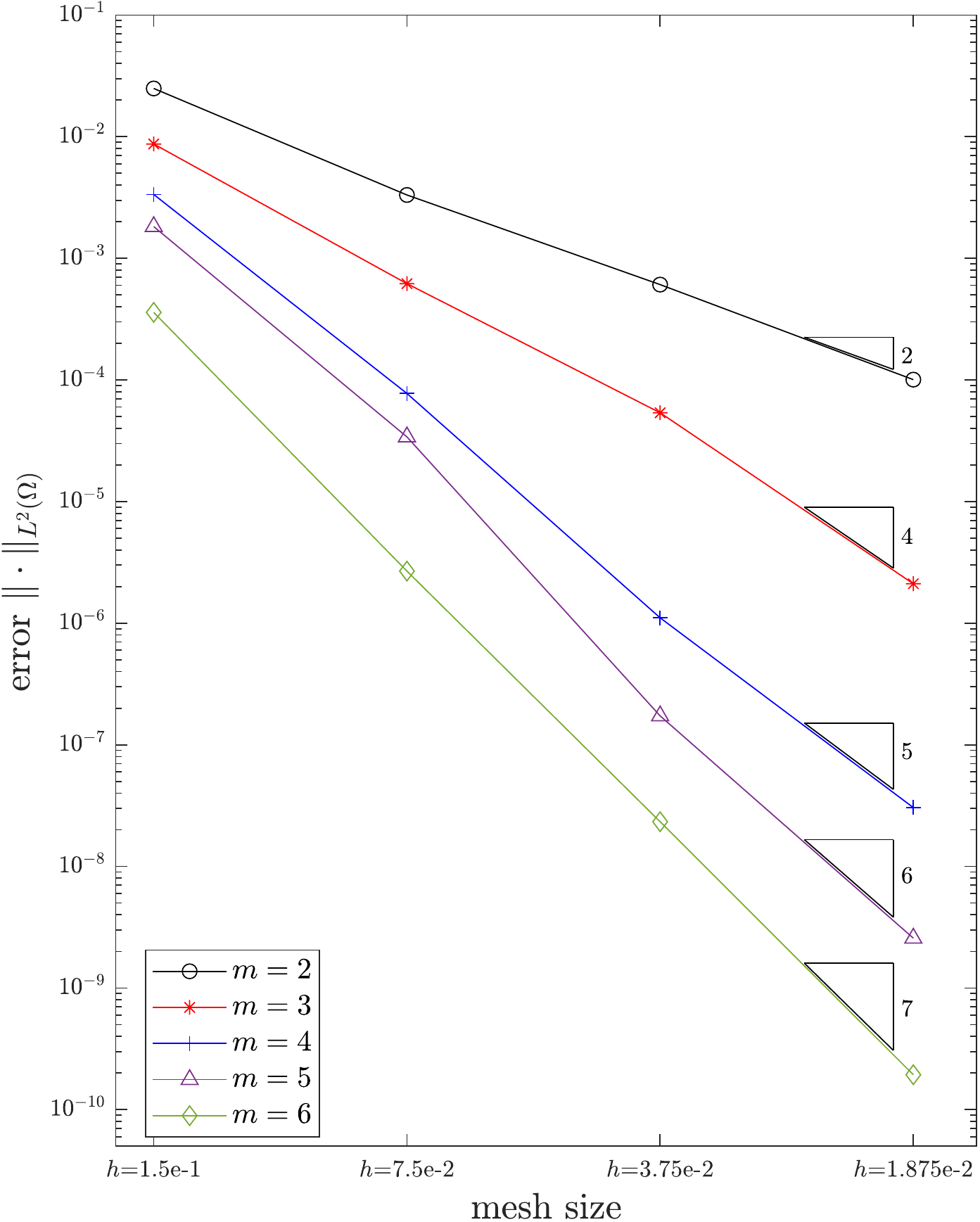}
  \caption{The convergence histories under the $\DGenorm{\cdot}$
  (left) and the $\| \cdot \|_{L^2(\Omega)}$ (right) in Example 2.}
  \label{fig_ex2err}
\end{figure}

\noindent \textbf{Example 3.}
In this example, we solve problem with a strong discontinuous
coefficient $\beta$. The problem is defined on the domain 
$\Omega = (-1,1)^2$ with an ellipse interface (see
Fig.~\ref{fig_ex3_interface}),
\begin{displaymath}
  \phi(x,y) = 2x^2 + 3y^2 - 1.
\end{displaymath}
The exact solution $u(x,y)$ and the coefficient $\beta$ is selected by
\begin{displaymath}
  u(x,y) = \left\{
  \begin{aligned}
    &\sin(2x^2 + y^2 + 2) + x, \quad (x,y) \in \Omega_0, \\
    & 0.1\cos(1-x^2-y^2), \quad (x,y) \in \Omega_1,
  \end{aligned}
  \right.
\end{displaymath}
\begin{displaymath}
  \beta = \left\{
  \begin{aligned}
    &1, \quad (x,y) \in \Omega_0, \\
    &100, \quad (x,y) \in \Omega_1.
  \end{aligned}
  \right.
\end{displaymath}
We have to use large penalty $\eta$ to handle the large jump in
$\beta$, see Tab.~\ref{tab_ex3_pen}. The numerical results presented 
in Fig.~\ref{fig_ex3err} demonstrates that the error $\DGenorm{u-u_h}$ 
and $\| u-u_h \|_{L^2(\Omega)}$ are still tends to zero with the rates 
we predicted in Theorem \eqref{th_errorestimate} and 
\eqref{th_L2errorestimate}. This example shows the robustness of the
proposed method.
\begin{table}[htp]
  \centering
  \begin{tabular}{p{0.6cm}|p{0.6cm}|p{0.6cm}|p{0.6cm}}
  \hline\hline
  $m$    & 2 & 3 & 4 \\ \hline
  $\eta$ & 50 & 100 & 300 \\ \hline
  $\# S$ & 12 & 18 & 25  \\ 
  \hline\hline
  \end{tabular}
\caption{The $\eta$ and $\# S$ used in Example 3..} 
\label{tab_ex3_pen}
\end{table}
\begin{figure}[htp]
  \centering
  \begin{minipage}[t]{0.46\textwidth}
    \centering
    \begin{tikzpicture}[scale=2]
      \centering
      \draw[thick, black] (-1, -1) rectangle (1, 1);
      \draw[thick, red, domain=0:360, samples=120] plot
      ({pow(2,-1/2)*cos(\x)}, {pow(3,-1/2)*sin(\x)});
    \end{tikzpicture}
  \end{minipage}
  \begin{minipage}[t]{0.46\textwidth}
    \centering
    \begin{tikzpicture}[scale=2]
      \centering
      \input{./figure/mth.tex}
      \draw[thick, red, domain=0:360, samples=120] plot
      ({pow(2,-1/2)*cos(\x)}, {pow(3,-1/2)*sin(\x)});
    \end{tikzpicture}
  \end{minipage}
  \caption{The interface and the mesh in Example 3.}
  \label{fig_ex3_interface}
\end{figure}
\begin{figure}[htbp]
  \centering
  \includegraphics[width=0.30\textwidth]{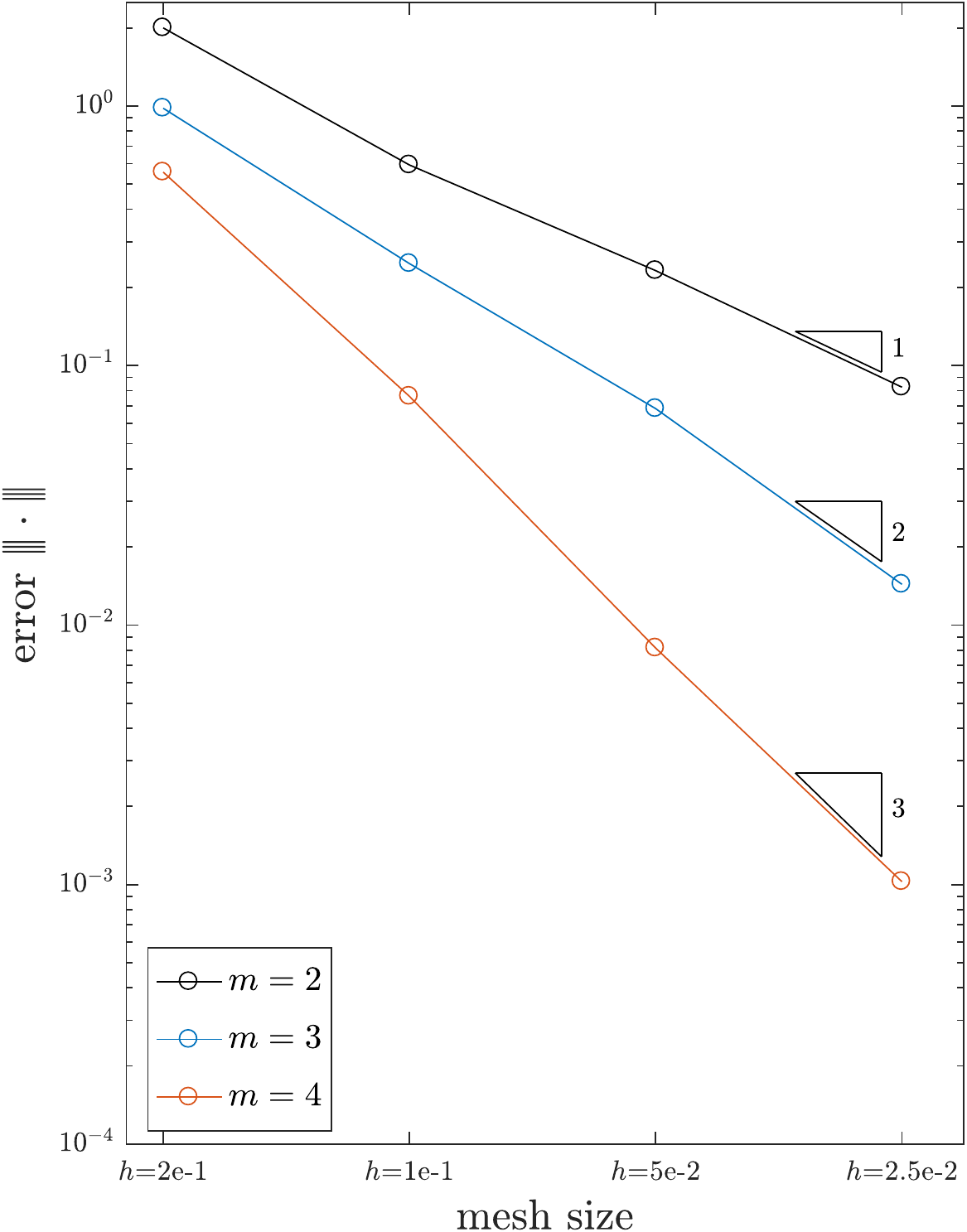}
  \hspace{30pt}
  \includegraphics[width=0.30\textwidth]{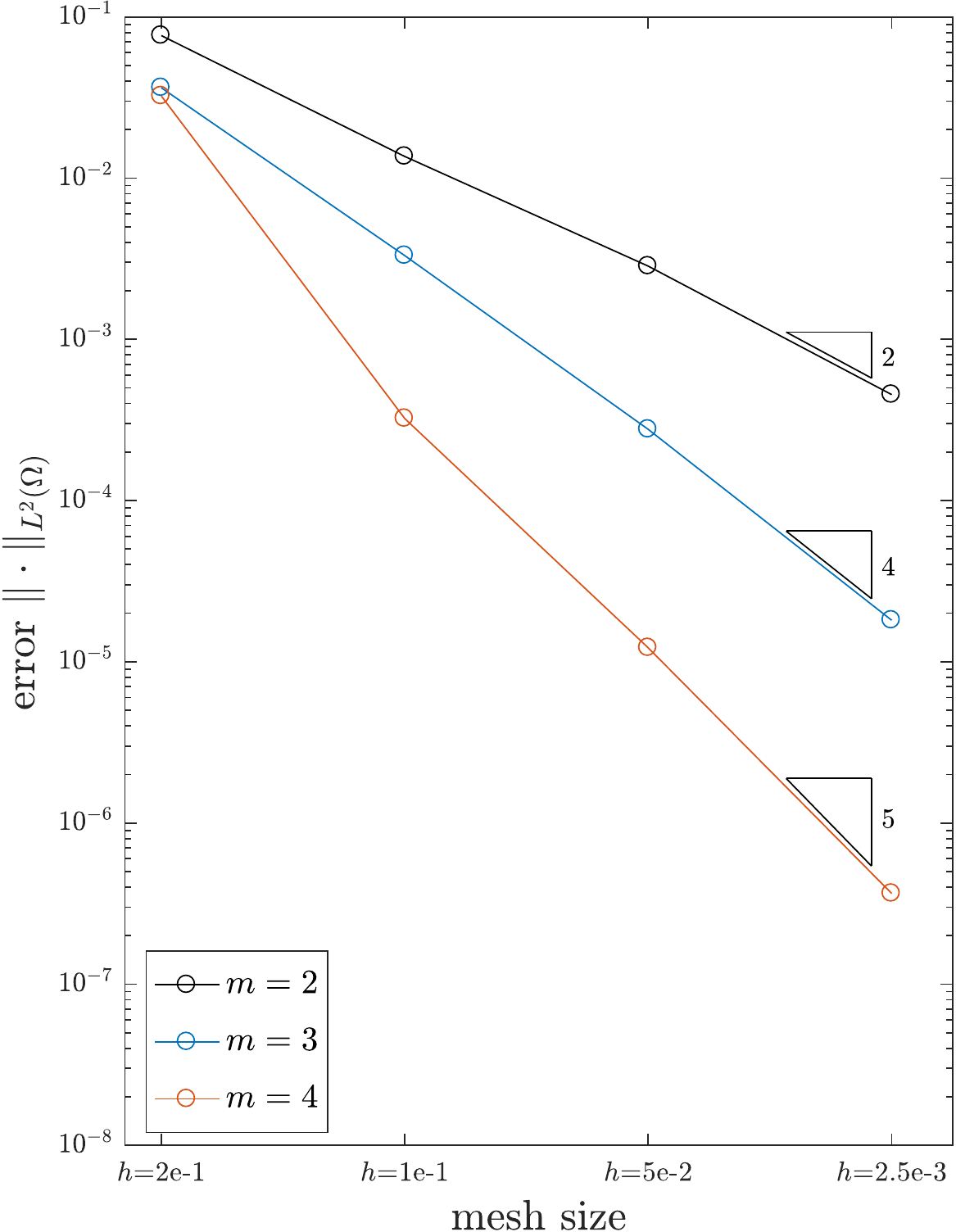}
  \caption{The convergence histories under the $\DGenorm{\cdot}$
  (left) and the $\| \cdot \|_{L^2(\Omega)}$ (right) in Example 3.}
  \label{fig_ex3err}
\end{figure}

\noindent \textbf{Example 4.}
In this example, we solve a problem on the domain $\Omega = (-1,1)^2$
with a five-pointed star shaped interface. The interface is given by
the polar coordinates (see Fig.~\ref{fig_ex4_interface}),
\begin{displaymath}
  r = \frac{1}{2} + \frac{\sin(5 \theta)}{7}.
\end{displaymath}
We choose the same analytical solution as Example 3 and the 
coefficient as
\begin{displaymath}
  \beta = \left\{
  \begin{aligned}
    &1, \quad (x,y) \in \Omega_0, \\
    &10, \quad (x,y) \in \Omega_1.
  \end{aligned}
  \right.
\end{displaymath}
We use the initial mesh size $h = 1/10$ and we successively refine 
the mesh three times for numerical tests. The convergence rates under
the energy norm and the $L^2$ norm are plotted in
Fig.~\ref{fig_ex3err}. We can observe from Fig.~\ref{fig_ex4err} that
the convergence rates under the $\DGenorm{\cdot}$ and the $\| \cdot
\|_{L^2(\Omega)}$ is $m-1$ and $m+1$ (except for the case $m=2$),
respectively, which is still consistent with the theoretical results.
\begin{figure}[htp]
  \centering
  \begin{minipage}[t]{0.46\textwidth}
    \centering
    \begin{tikzpicture}[scale=2]
      \centering
      \draw[thick, black] (-1, -1) rectangle (1, 1);
      \draw[thick, domain=0:360, red, samples=120] plot (\x:{(0.5 +
      sin(\x*5)/7)*1.02});
    \end{tikzpicture}
  \end{minipage}
  \begin{minipage}[t]{0.46\textwidth}
    \centering
    \begin{tikzpicture}[scale=2]
      \centering
      \input{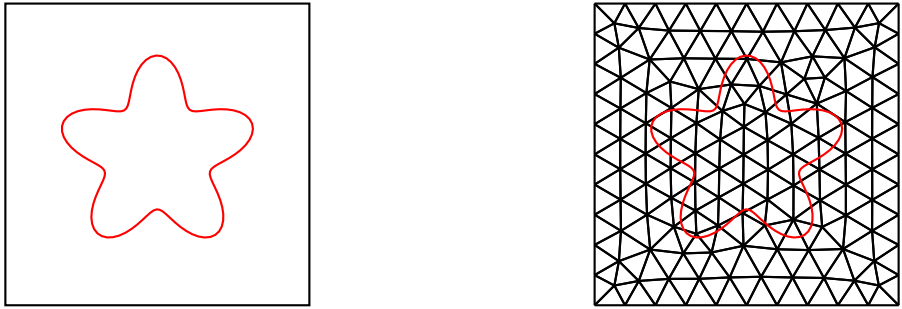}
      \draw[thick, domain=0:360, red, samples=120] plot (\x:{(0.5 +
      sin(\x*5)/7)*1.02});
    \end{tikzpicture}
  \end{minipage}
  \caption{The interface and the mesh in Example 4.}
  \label{fig_ex4_interface}
\end{figure}
\begin{figure}[htbp]
  \centering
  \includegraphics[width=0.30\textwidth]{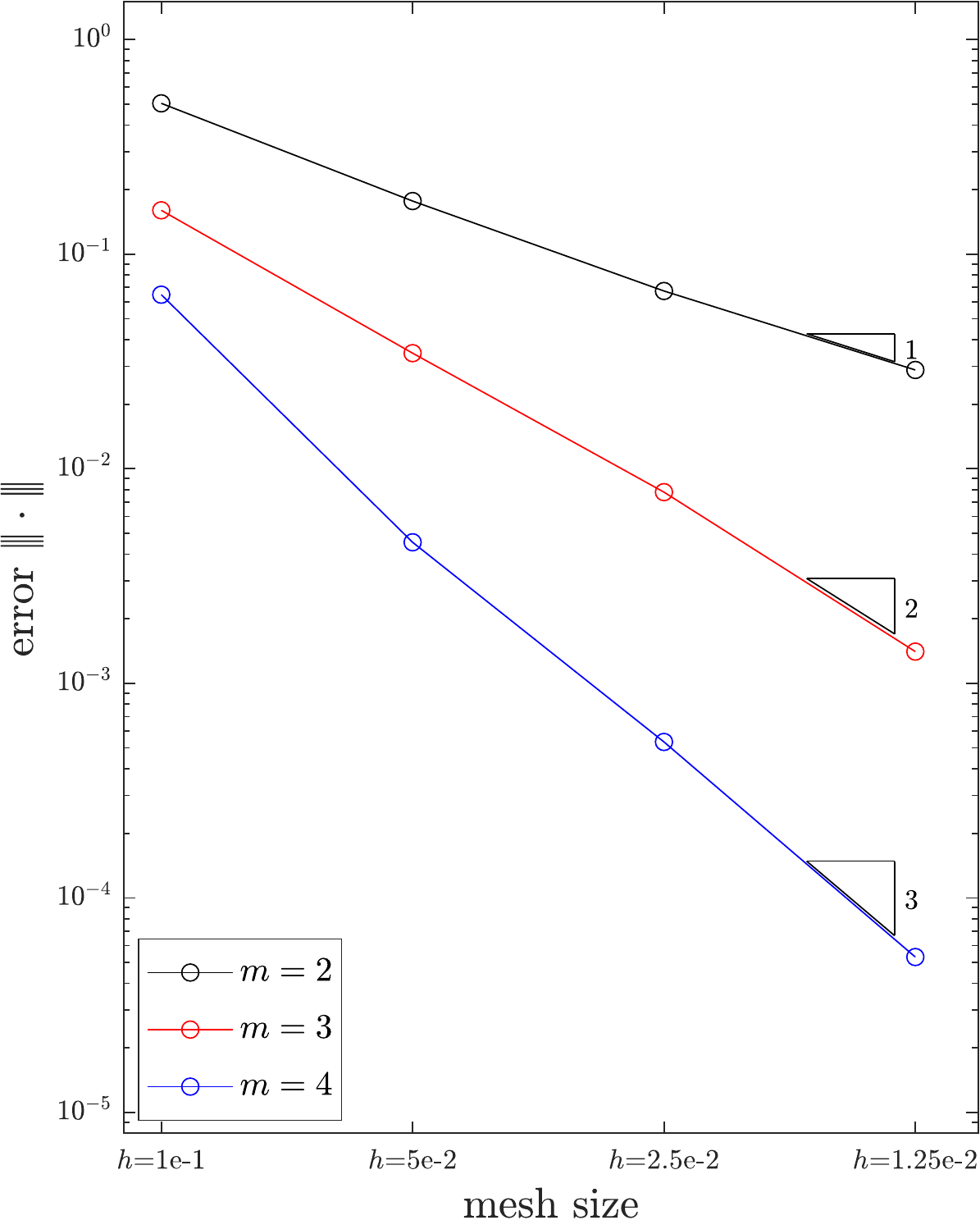}
  \hspace{30pt}
  \includegraphics[width=0.30\textwidth]{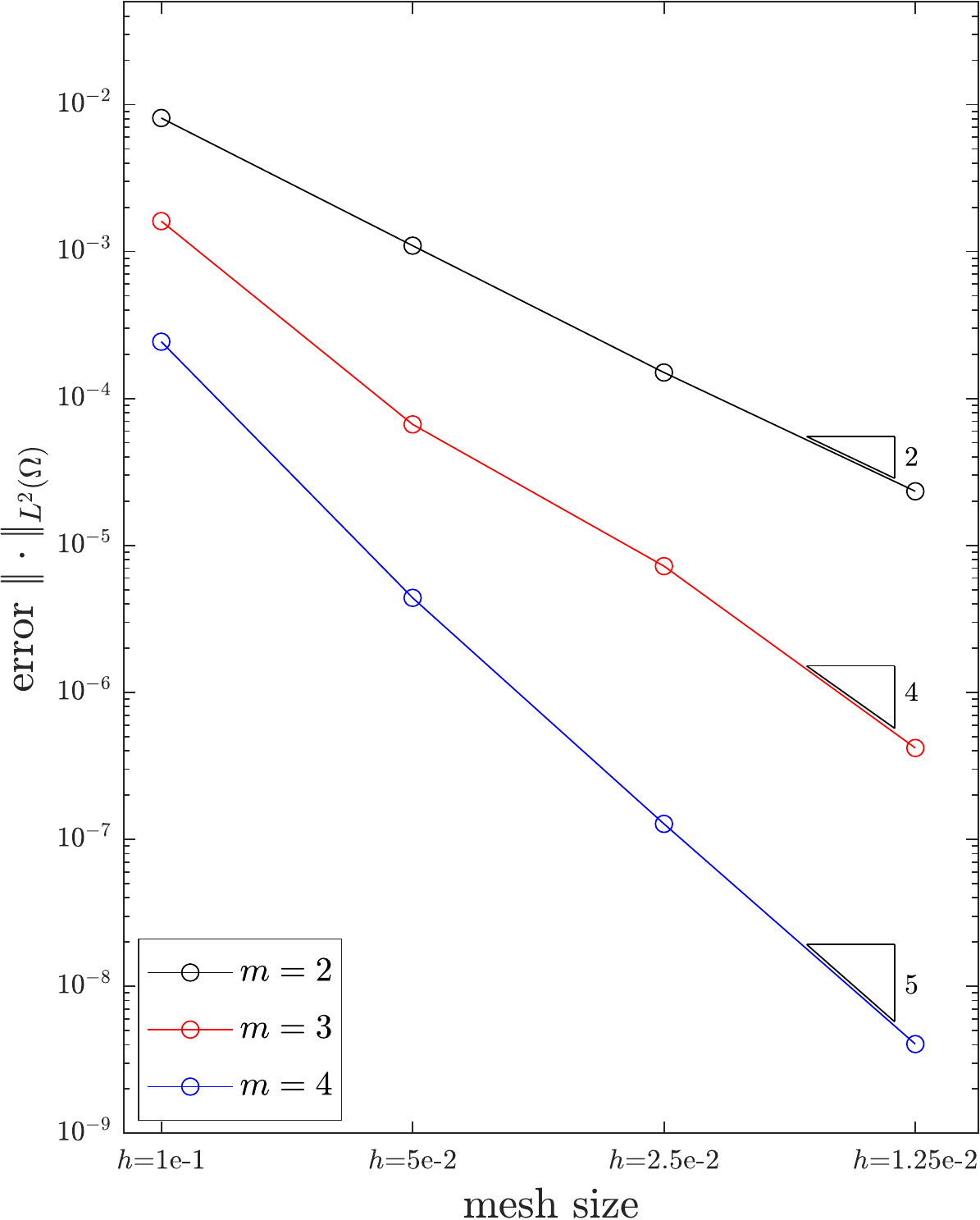}
  \caption{The convergence histories under the $\DGenorm{\cdot}$
  (left) and the $\| \cdot \|_{L^2(\Omega)}$ (right) in Example 4.}
  \label{fig_ex4err}
\end{figure}

\noindent \textbf{Example 5.}
We solve a three-dimensional biharmonic interface problem in this
case. The computation domain is an unit cube $\Omega = (0,1)^3$
containing a spherical interface in its interior (see
Fig.~\ref{fig_ex5_interface}),
\begin{displaymath}
  \phi(x,y,z) = (x-0.5)^2 + (y-0.5)^2 + (z-0.5)^2 - r^2, \quad r =
  0.35.
\end{displaymath}
We select the coefficient and the exact solution as
\begin{displaymath}
  \beta = \left\{
  \begin{aligned}
    &1, \quad \text{ inside } \Gamma, \\
    &2, \quad \text{ outside } \Gamma,
  \end{aligned}
  \right.
\end{displaymath}
\begin{displaymath}
  u(x,y,z) = \left\{
  \begin{aligned}
    &\cos(x^2 + y^2 + z^2), \quad \text{ inside } \Gamma, \\
    &\sin(x)\sin(y)\sin(z), \quad \text{ outside } \Gamma.
  \end{aligned}
  \right.
\end{displaymath}
We use the tetrahedra meshes generated by the Gmsh software
\cite{geuzaine2009gmsh}. We solve the problem on five different meshes
with the reconstruction order $m = 2,3$. The relationship between the
cubic root of degrees of freedom and errors is shown in
Fig.~\ref{fig_ex5err}, which is also clearly consistent with our
theoretical predictions.
\begin{figure}[htp]
  \centering
  \begin{minipage}[t]{0.46\textwidth}
    \centering
    \begin{tikzpicture}[scale=3]
      \centering
      \draw[very thin, dashed,black] (0,0,0) -- (1,0,0);
      \draw[very thin, dashed,black] (0,0,0) -- (0,1,0);
      \draw[very thin, dashed,black] (0,0,0) -- (0,0,1);
      \draw[very thin, black] (0,1,1) -- (1,1,1);
      \draw[very thin, black] (0,1,1) -- (0,1,0);
      \draw[very thin, black] (0,1,1) -- (0,0,1);
      \draw[very thin, black] (1,0,1) -- (0,0,1);
      \draw[very thin, black] (1,0,1) -- (1,1,1);
      \draw[very thin, black] (1,0,1) -- (1,0,0);
      \draw[very thin, black] (1,1,1) -- (1,1,0);
      \draw[very thin, black] (1,1,0) -- (1,0,0);
      \draw[very thin, black] (1,1,0) -- (0,1,0);
      \fill[ball color=red, opacity=0.6] (0.5,0.5,0.5) circle (0.35);
    \end{tikzpicture}
  \end{minipage}
  \hspace{-2cm}
  \begin{minipage}[t]{0.46\textwidth}
    \centering
    \begin{tikzpicture}[scale=3]
      \centering
      \input{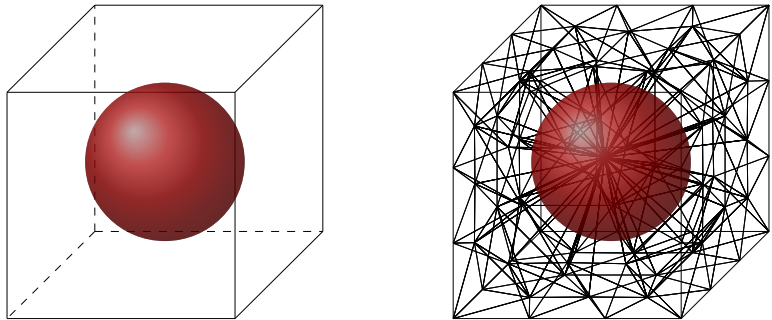}
      \fill[ball color=red, opacity=0.6] (0.5,0.5,0.5) circle (0.35);
    \end{tikzpicture}
  \end{minipage}
  \caption{The interface and the mesh in Example 5.}
  \label{fig_ex5_interface}
\end{figure}
\begin{figure}[htbp]
  \centering
  \includegraphics[width=0.30\textwidth]{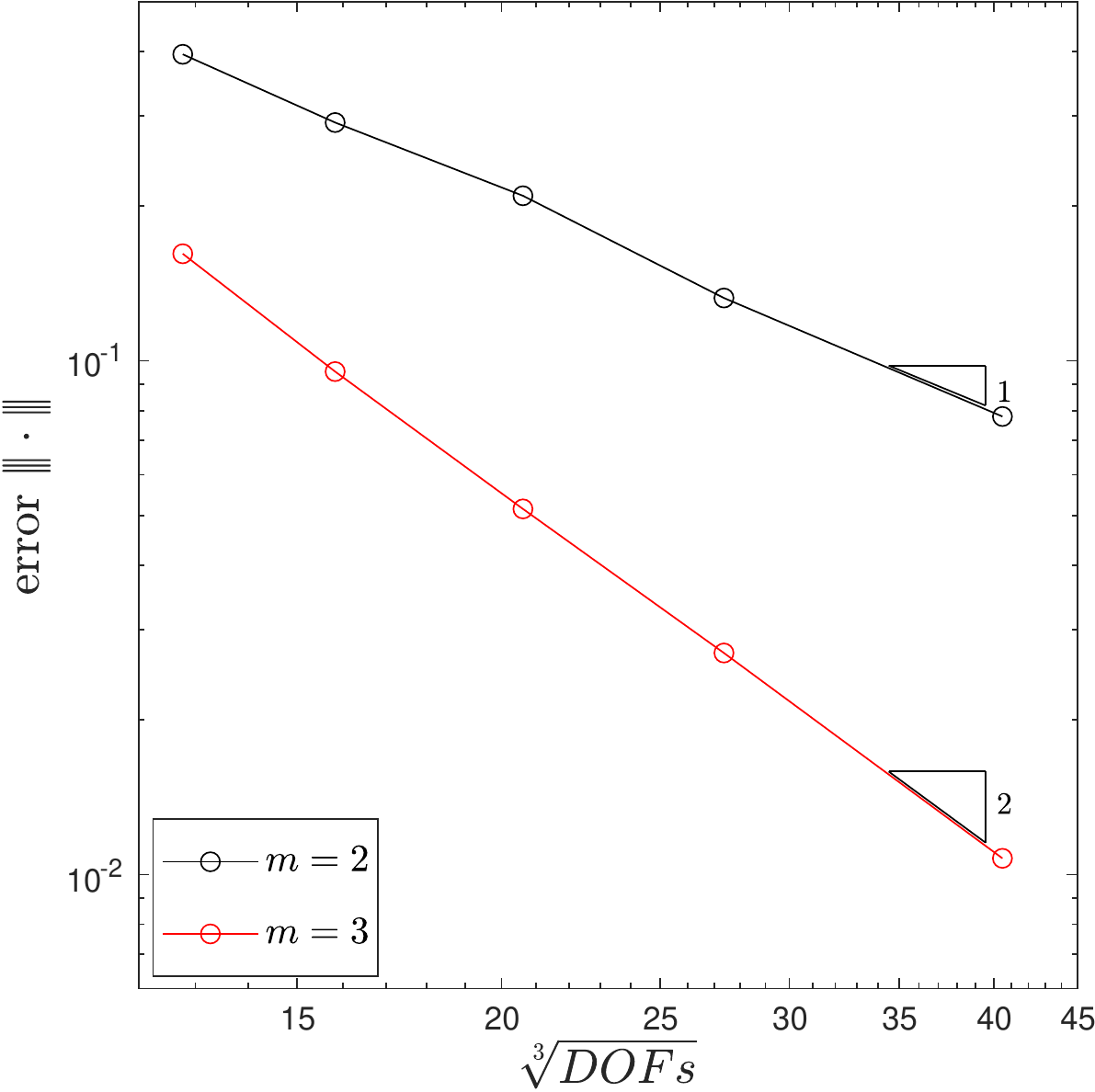}
  \hspace{30pt}
  \includegraphics[width=0.30\textwidth]{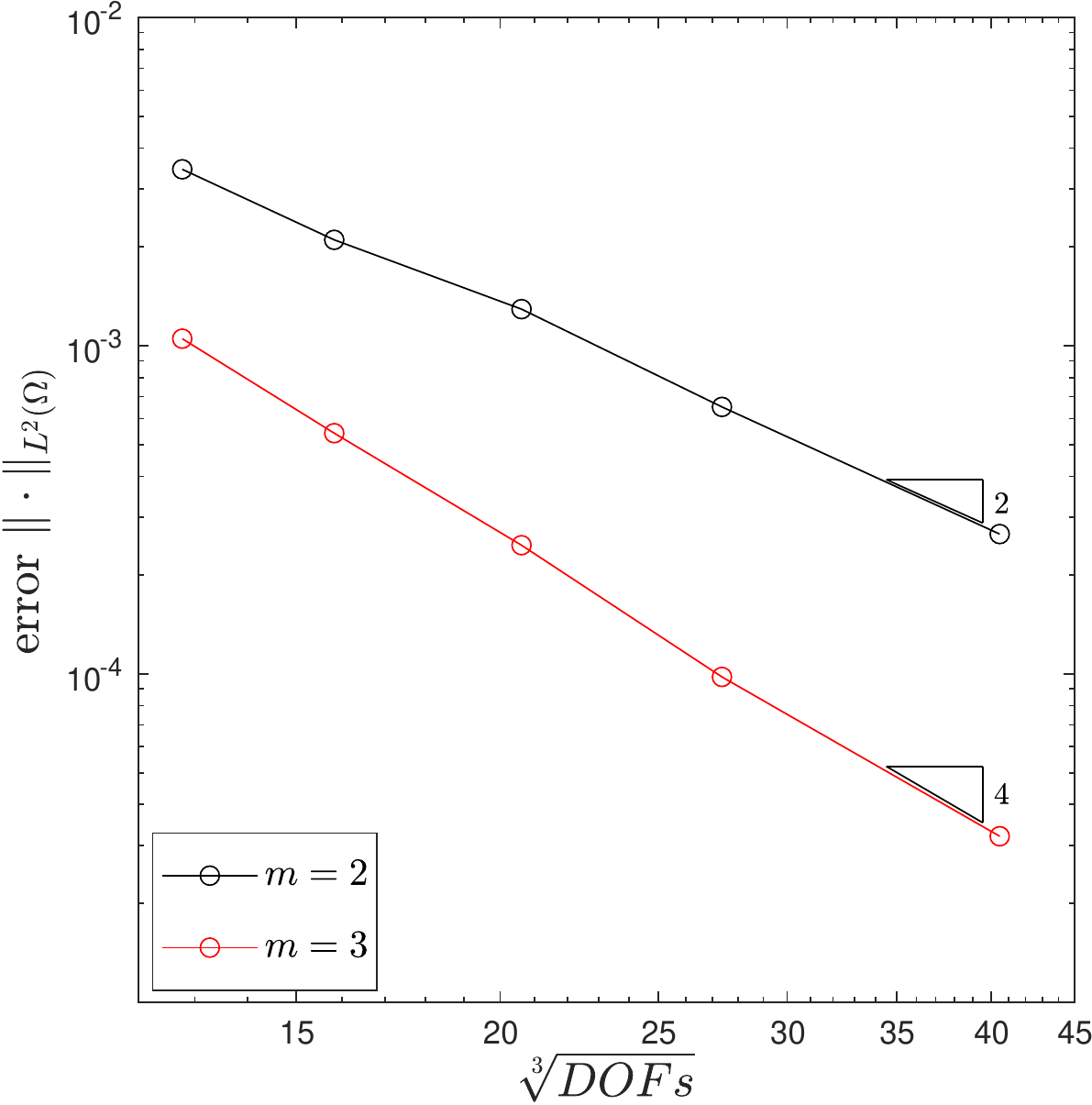}
  \caption{The convergence histories under the $\DGenorm{\cdot}$
  (left) and the $\| \cdot \|_{L^2(\Omega)}$ (right) in Example 5.}
  \label{fig_ex5err}
\end{figure}

\section{Conclusion}
\label{sec_conclusion}
In this paper, we propose an arbitrary order discontinuous Galerkin
extended finite element method for solving the biharmonic interface 
problem. The discrete formulation is obtained by a symmetric interior 
penalty method and the jump condition is enforced by Nitsche's trick
in a weak sense. The approximation space is constructed by a patch
reconstruction operator and the number of degrees of freedom is
independent of the approximation order. Our method is easily
implemented and can achieve high-order accuracy. It is shown the
optimal convergence rates for the numerical errors under the energy
norm and the $L^2$ norm. We present a series of numerical experiments
to verify the theoretical results and the efficiency of the proposed
method.

\bibliographystyle{amsplain}
\bibliography{../ref}

\end{document}